\documentclass[11pt]{amsart}
 \usepackage{amsmath,amsthm,amssymb,amsfonts,graphicx,url,hyperref,xcolor,pinlabel}

\newtheorem{theorem}{Theorem}[section]
\newtheorem{proposition}[theorem]{Proposition}
\theoremstyle{definition}
\newtheorem{definition}[theorem]{Definition}
\newtheorem{remark}[theorem]{Remark}
\theoremstyle{definition}
\newtheorem{example}[theorem]{Example}
\newtheorem{notation}[theorem]{Notation}
\newtheorem{convention}[theorem]{Convention}

\title[Describing fibered 3-manifolds]{Explicitly describing fibered 3-manifolds through families of singularly fibered surfaces}
\author{Maggie Miller}
\address{Department of Mathematics\\Stanford University\\Stanford, CA 94301, USA}
\email{maggie.miller.math@gmail.com}

\thanks{The author is supported by a Clay Research Fellowship and a Stanford Science Fellowship.}

\begin{document}
\maketitle

\begin{abstract}
We give an explicit description of a fibration of the complement of the closure of a homogeneous braid, understanding how each fiber intersects every cross-section of $S^3$.
\end{abstract}

\section{Introduction}

In this note, we give an explicit description of a fibration of the complement of the closure of a homogeneous braid (e.g.\ a torus knot). These links have long been known to be fibered, since Stallings \cite{stallings} observed that a standard Seifert surface for such a link decomposes into a Murasugi sum of fiber surfaces for torus knots.


Abstractly proving the existence of a fibration on a knot or link complement is generally straightforward, since (for example) fiberedness is deteected by knot or link Floer homology \cite{ni}. However, the author of this paper usually finds it difficult to conceptualize the whole fibration of such a link complement in total, rather than a single fiber. While building fibrations in a 4-dimensional setting, we found that a deeper understanding of fibrations of classical knot complements seemed necessary in order to extrapolate to higher dimensions. We thus worked to explicitly describe how the fibers of certain link complements meet 
every cross-section of a standard Morse function on $S^3$.

We first give an imprecise statement of the main theorem, meant to illustrate the aim of the paper.
\begin{theorem}\label{thm:mainfirst}
Let $L$ be the closure of a homogenous braid $\beta$. 
Then $L$ is a fibered link and we can explicitly describe a fibration $F:S^3\setminus\nu(L)\to S^1$. The level sets of $F$ are simple, in the sense 
 that there is a natural Morse function $h$ on $(S^3, L)$ so that for every $\theta\in S^1$, the restriction of $h$ to the interior of $F^{-1}(\theta)$ has no local minima or maxima. 
\end{theorem}

The proof of Theorem \ref{thm:mainfirst} is an explicit construction, which we are sharing with the belief that this can motivate arguments in 4-dimensional (or higher) manifolds where the usual techniques in 3-dimensional topology fail. (In the author's mind, this paper is mostly an exposition on the structure of a fibration of a 3-manifold over $S^1$.) We focus on closures of homogeneous braids as a natural, well-understood family of fibered links. We obtain a very restrictive statement about the position of the fibers of such a link that might be of interest to knot theorists and 3-dimensional topologists. 

We give the precise statement of Theorem \ref{thm:mainfirst} now, which we will prove constructively in the following sections of this paper.

\begin{theorem}\label{thm:main}
Let $L$ be the closure of a homogenous braid $\beta$ on $b$ strands, braided with respect to the standard Morse function $h:S^3\to\mathbb{R}$, with the minima and maxima of $h$ also minima and maxima of $h|_L$. 
 Then there is a fibration $F:S^3\setminus\nu(L)\to S^1$ so that 
 simultaneously for all $\theta$, the restriction of $h$ to the interior of $F^{-1}(\theta)$ is Morse with no local minima nor maxima and the restriction of $h$ to the boundary of $F^{-1}(\theta)$ is Morse with $b$ local minima and $b$ local maxima.
\end{theorem}

\section{Braid conventions}

We choose some conventions when discussing links in braid position.
\begin{definition}
Decompose $S^3=B^3_+\cup(S^2\times I)\cup B^3_-$. Let $h:S^3\to\mathbb{R}$ be Morse so that $h$ has one local maximum in $B^3_+$, one local minimum in $B^3_-$ and the restriction $h|_{S^2\times I}$ is projection to $I$.

Further decompose $S^2\times I=(D^2_L\times I)\cup(D^2_R\times I)$. Let $\beta$ be a braid in $D^2_L\times I$ on $b$ strands. Let $L$ be a link with
\[
L\cap (D^2_L\times I)=\beta\qquad\text{and}\qquad
L\cap (D^2_R\times I)=\{b\text{ points}\}\times I,\]
and so that $L\cap B^3_+$ and $L\cap B^3_-$ each consist of $b$ strands, on each of which $h$ has one local critical point, with two of these extrema being the two critical points of $h$. We say $L$ is in {\emph{braid position with respect to $h$}} and is the {\emph{braid closure of $\beta$}}. See Figure \ref{fig:braid}.
\end{definition}

\begin{figure}
\centering
\labellist
\pinlabel{\textcolor{gray}{$B^3_-$}} at 0 70
\pinlabel{\textcolor{gray}{$D^2_L\times I$}} at -25 200
\pinlabel{\textcolor{gray}{$D^2_R\times I$}} at 325 200
\pinlabel{\textcolor{gray}{$B^3_+$}} at 0 330
\pinlabel{$\beta$} at 20 240
\endlabellist
\includegraphics[width=50mm]{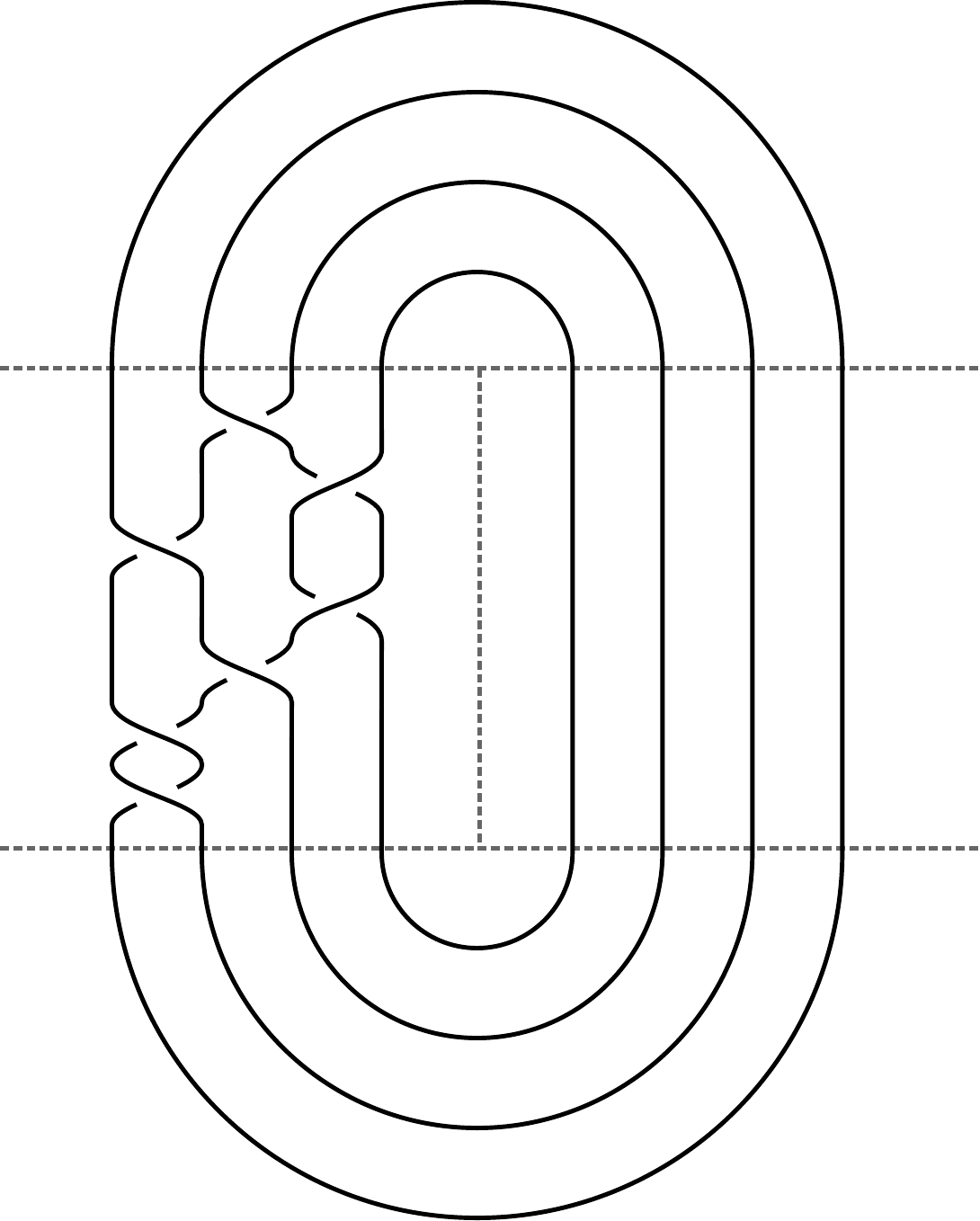}
\caption{A link $L$ in braid position as the braid closure of $\beta$. Here, $\beta$ is a homogeneous braid since its given diagram is described by the braid word $\sigma_1^{-1}\sigma_1^{-1}\sigma_2^{-1}\sigma_3\sigma_1^{-1}\sigma_3\sigma_2^{-1}$, which is a word in the letters $\sigma_1^{-1},\sigma_2^{-1},\sigma_3$ with each letter appearing at least once.}\label{fig:braid}
\end{figure}
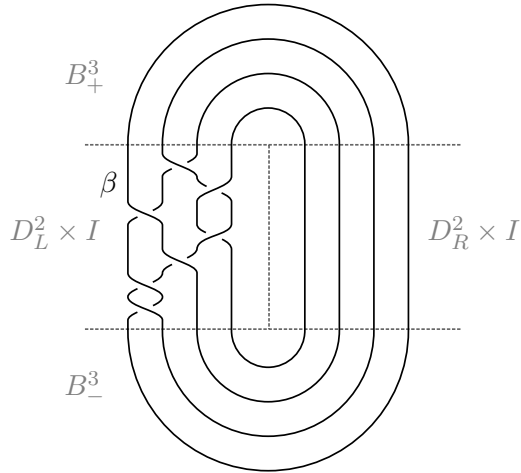

Recall that a braid diagram of a braid $\beta$ on $b$ strands can be written as a word in $\sigma_1,\sigma_1^{-1},\ldots,\sigma_{b-1},\sigma_{b-1}^{-1}$, where $\sigma_i$ is a diagram involving one positive half-twist between the $i$-th and $(i+1)$-th strands and no other crossings, and $\sigma_i^{-1}$ is similar but with a negative half-twist. We will use the convention that the letters in a word $w$ from left to right describe crossings in a diagram of $\beta$ from bottom to top. This means that if $w_1,w_2$ are words for $\beta_1,\beta_2$, then $w_1w_2$ is the word for the braid obtained by stacking $\beta_2$ above $\beta_1$.


\begin{definition}
A braid is {\emph{homogeneous}} if it can be written as a word in $\sigma_1^{s_1},$ $\ldots,$ $\sigma_{b-1}^{s_{b-1}}$ for some $s_1,\ldots, s_{b-1}\in\{-1,1\}$, with each letter appearing at least once.

In simpler terms, a braid is homogeneous if it can be written as a braid word so that for each $i$, the word includes one of $\sigma_i$ and $\sigma_i^{-1}$, but not both. See Figure \ref{fig:braid}.
\end{definition}


\section{Background: constructing fibrations}
In this section, we discuss local constructions of fibrations in dimension three. Our plan is to fix a height function $h:M^3\to\mathbb{R}$ whose regular level sets are surfaces and whose singular level sets are surfaces with well-understood singularities. (In the case that $M$ is closed, $h$ is simply a Morse function.) We will then construct a family of functions $f_t$ mapping each level set $h^{-1}(t)$ to the circle, with rules on how $f_t$ changes with $t$ to ensure that the function $F:M\to S^1$ given by $F(x)=f_{h(x)}(x)$ is a fibration. The advantage of this approach is that while drawing a collection of surfaces in a 3-manifold may seem daunting, drawing a collection of curves (level sets of $f_t$) in a surface (a level set of $h$) may seem more approachable.


\begin{convention}
In this paper, we use the terms ``singular fibration" and ``height function" rather than ``circular Morse function" or ``Morse function" on manifolds with boundary to make it clear that we do not expect functions to be locally constant on boundary. For us, a {\emph{height function}} $h:M\to\mathbb{R}$ is a smooth function that has the following properties.
\begin{itemize}
\item $h|_{\partial M}$ is Morse.
\item $h$ is Morse on the interior of $M$.
\item The level sets of $h$ meet $\partial M$ transversely away from critical points of $h|_{\partial M}$.
\end{itemize}
 For example, if $L$ is a link in $S^3$ and $h:S^3\to\mathbb{R}$ is Morse with  $h|_L$ also Morse, then $h$ restricts to a height function on $M:=S^3\setminus\nu(L)$ for $\nu(L)$ a suitable tubular neighborhood of $L$. 
\end{convention}

We first observe that given a fibration $F:M^3\to S^1$ and a height function $h:M^3\to\mathbb{R}$, we generically expect $F$ to induce a singular fibration on each level set of $h$.
\begin{definition}\label{def:singularfibsurface}
A {\emph{singular fibration}} on a compact surface $\Sigma$ is a smooth map $f:\Sigma\to S^1$ with the following properties.
\begin{itemize}
\item The restriction $f|_{\partial\Sigma}$ is Morse.
\item For all but finitely many $\theta$, $f^{-1}(\theta)$ is a compact, properly embedded 1-manifold.
\item For all $\theta$, $f^{-1}(\theta)$ is a compact, properly embedded 1-manifold away from a finite number of singularities of the four types shown in Figure \ref{fig:leafsingularity}. We will often refer to interior wedgepoints (first image in Figure \ref{fig:leafsingularity}) as {\emph{$\times$-singularities}}.
\end{itemize}

We refer to each $f^{-1}(\theta)$ as a {\emph{leaf}} of $f$.
\end{definition}

\begin{figure}
\centering
\labellist
\endlabellist
\includegraphics[width=100mm]{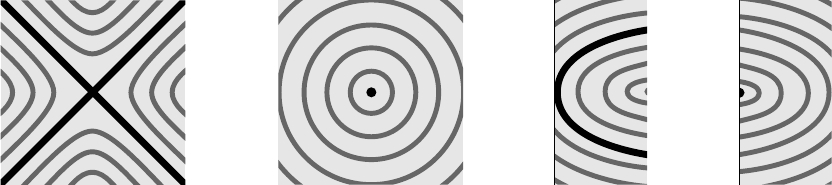}
\caption{Local models near singular points of a leaf of a singular fibration $f:\Sigma\to S^1$ (away from singularities of $\Sigma$). {\bf{First:}} A leaf includes an $\times$-singularity. {\bf{Second:}} A leaf component is a single point in the interior of $\Sigma$.
{\bf{Third:}} A leaf has an isolated tangency to $\partial\Sigma$.
{\bf{Fourth:}} A leaf component is a single point in the boundary of $\Sigma$.
{}}\label{fig:leafsingularity}
\end{figure}

We may extend Definition \ref{def:singularfibsurface} to apply to some singular surfaces. Here, by {\emph{singular surface}} we mean a compact topological space embeddable in $\mathbb{R}^4$ containing a finite singularity set $S$ whose complement is a surface. 

\begin{definition}\label{def:singularfibration}
Let $\Sigma$ be a compact, singular surface with finite singularity set $S=\{p_1,\ldots, p_m,q_1,\ldots, q_n\}$. 

In words, the points $p_1,\ldots, p_m$ are the singularities on the boundary of $\Sigma$. These points come in three types: some are seemingly interior points that we have artificially declared to be boundary, some are wedge points in $\partial\Sigma$, and others are isolated points that we declare to be in the boundary of $\Sigma$. The points $q_1,\ldots, q_n$ are the singularities in the interior of $\Sigma$. These points come in two types: the singular points of double cones and isolated points that we declare to be in the interior of $\Sigma$.

To be precise, we require each $p_i$ to have a neighborhood that is homeomorphic to one of: $\mathbb{R}^2$ or $\{(x,y)\mid|x|\ge|y|\}$ or $\{(0,0)\}$,  with $p_i\mapsto(0,0)$. We require each $q_i$ to have a neighborhood that is homeomorphic to one of: $\{(x,y,z)\mid x^2+y^2=|z|\}$ or $(0,0)$, with $q_i\mapsto(0,0)$.



A smooth (away from $S$) function $f:\Sigma\to S^1$ is a {\emph{singular fibration}} if $f$ restricts to a singular fibration on $\Sigma\setminus\nu(S)$ as in Definition \ref{def:singularfibsurface} and in each component of $\nu(S)$, $f$ restricts to one of the local models shown in Figure \ref{fig:singularsurfaces}.

\end{definition}

\begin{figure}
\centering
\labellist
\pinlabel{boundary singularities} at 170 250
\pinlabel{interior singularities} at 170 100
\endlabellist
\vspace{.2in}
\includegraphics[width=90mm]{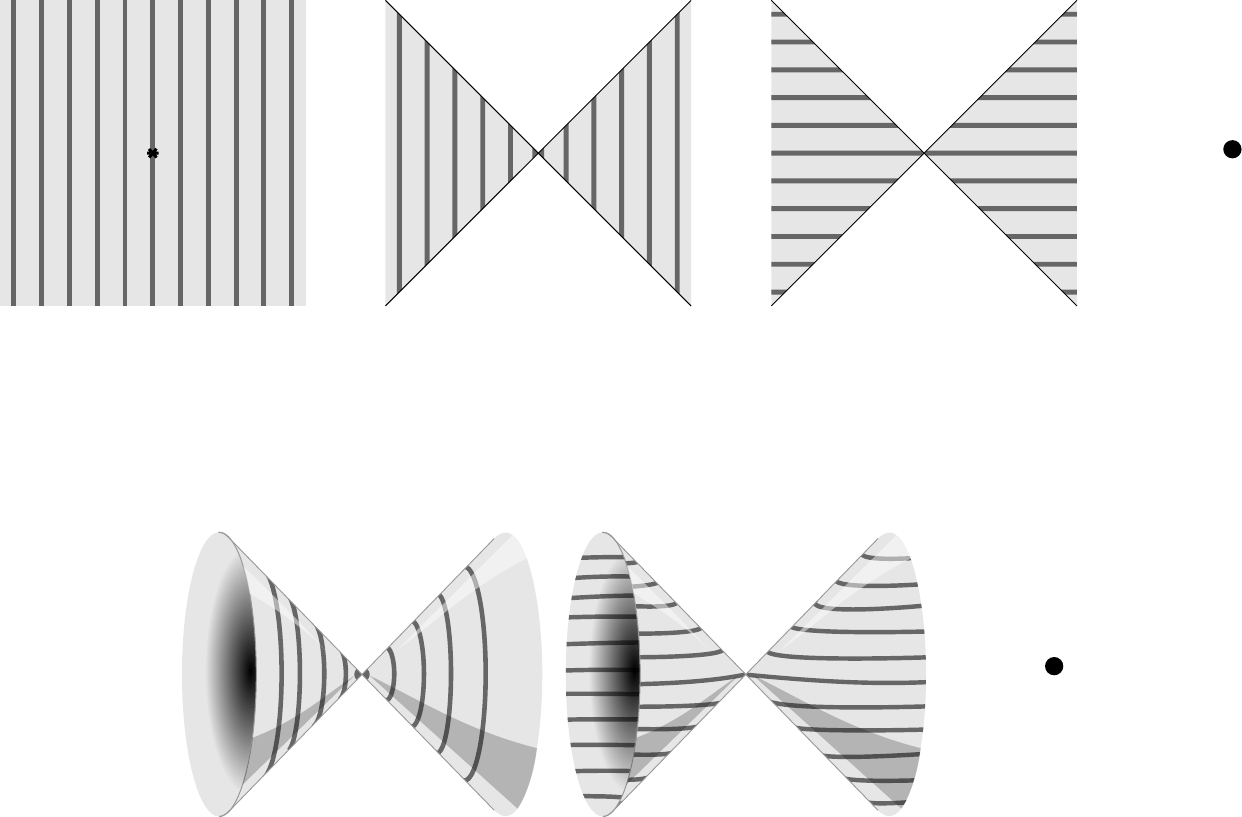}
\caption{Near a singular point of a surface $\Sigma$, we require the leaves of a singular fibration $f$ to have one of these pictured local models.}\label{fig:singularsurfaces}
\end{figure}

In Definition \ref{def:singularfibration}, we restrict to the given types of singularities precisely because these are the singularities that appear in level sets of a height function $h:M\to\mathbb{R}$. Interior singular points arise in $h^{-1}(t)$ when the Morse function $h|_{\mathring{M}}$ has a critical point at height $t$. Boundary singular points arise in $h^{-1}(t)$ when the Morse function $h|_{\partial M}$ has a critical point at height $t$.



We use singular fibrations of surfaces to construct a fibration of a 3-manifold by considering the level surfaces of a height function.

\begin{definition}
Given a height function $h:M^3\to\mathbb{R}$, a {\emph{movie of singular fibrations}} on $M^3$ is a family of smooth maps $f_t:h^{-1}(t)\to S^1$ with each $f_t$ a singular fibration on the singular surface $h^{-1}(t)$ such that the functions $f_t$ vary smoothly with $t$, i.e.\ $F(x)=f_{h(x)}(x)$ is a smooth map from $M^3$ to $S^1$. We refer to $F$ as the {\emph{total function}} of the movie $f_t$.
\end{definition}

By constructing a movie of singular fibrations on $M^3$ and keeping track of how singular leaves of $f_t$ vary as $t$ increases or decreases, we can arrange for the total function $F:M\to S^1$ to be a fibration. In \cite{miller} we referred to the ``type" of singularities in leaves of $f_t$, determined by the sign of $\frac{d}{dt}f_t$ at that singularity. This language is less useful in this dimension due to symmetry of $\times$-singularities. (Or in other words, this notation is less useful because an index-1 critical point of a Morse function on a surface remains an index-1 point when turned upside down.) Instead, we add arrow decorations to contour maps of $f_t$, as shown in Figures \ref{fig:arrows} and \ref{fig:singulararrows}. In Figure \ref{fig:singulararrows} we draw only boundary singularities; models of interior singularities can be obtained by doubling. 

\begin{figure}
\centering
\labellist
\pinlabel{$t$} at -20 80
\pinlabel{$t$} at 85 305
\endlabellist
\includegraphics[width=126mm]{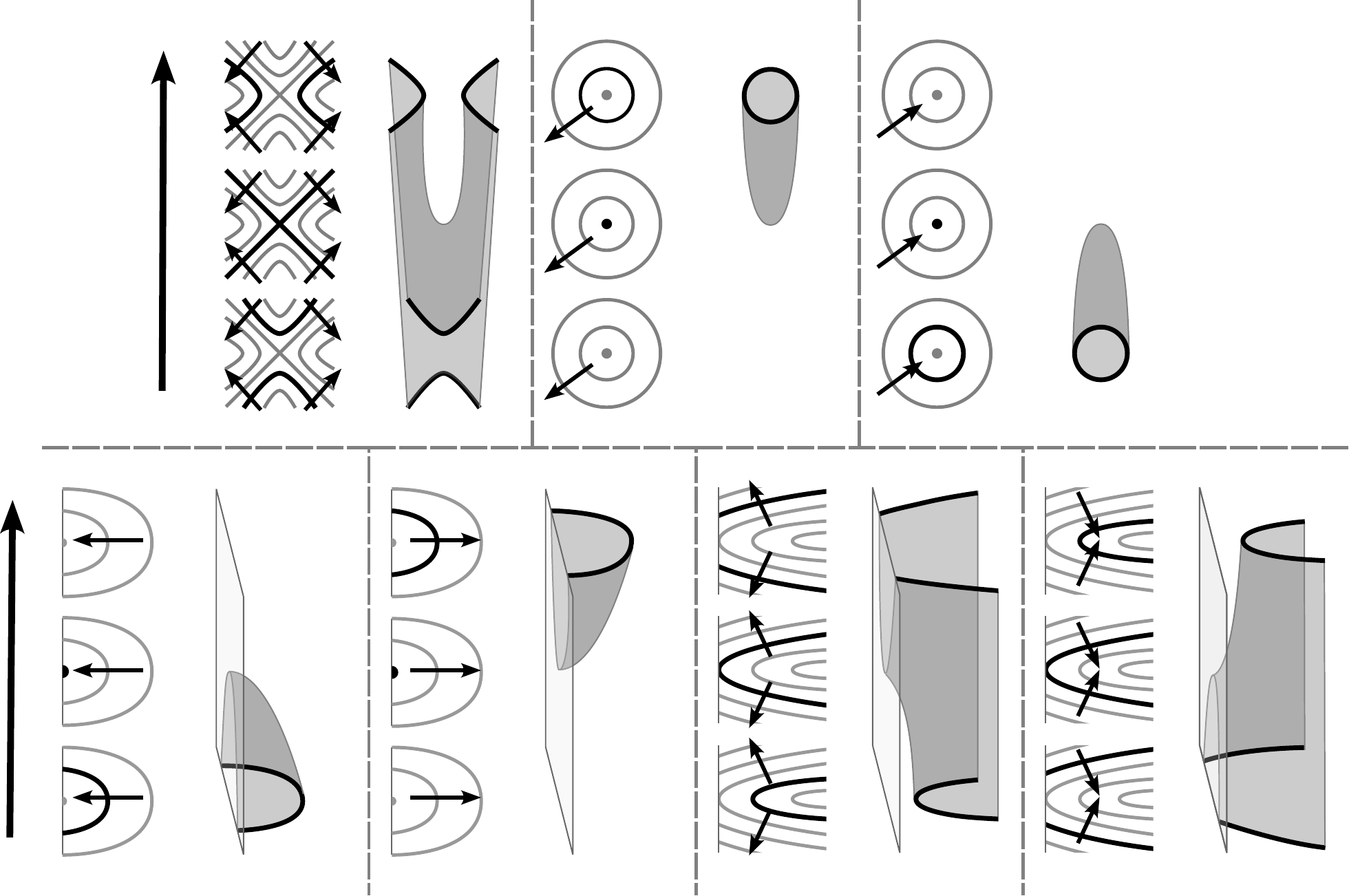}
\caption{Local models near singularities of leaves of $f_t$ that are not also singular points of $h^{-1}(t)$. In all of these models, $\frac{d}{dt}f_t$ is nonzero near the singular point. In each cell, we highlight cross-sections of $F^{-1}(\theta)$ as $t$ increases and give a schematic of $F^{-1}(\theta)$. We include arrows near singularities of leaves of $f_t$ that indicate the behavior of $f_{t\pm\epsilon}$ near the singularity.
}\label{fig:arrows}
\end{figure}

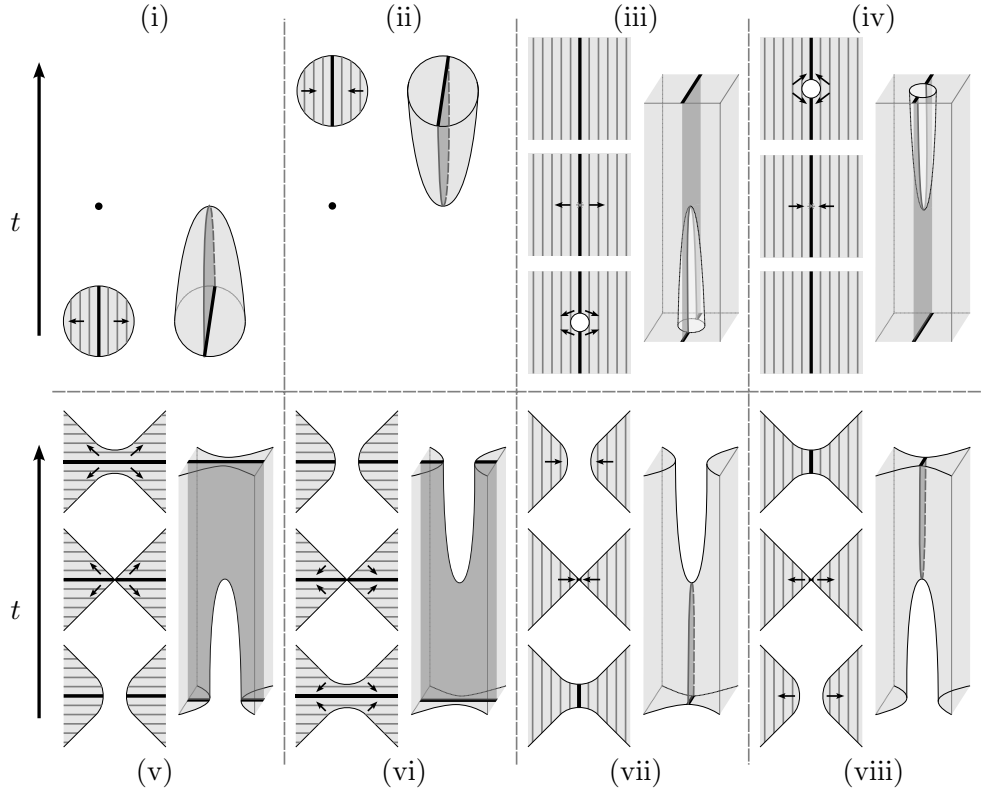
\begin{figure}
\centering
\labellist
\pinlabel{$t$} at -15 150
\pinlabel{$t$} at -15 485
\pinlabel{(i)} at 105 660
\pinlabel{(ii)} at 320 660
\pinlabel{(iii)} at 520 660
\pinlabel{(iv)} at 725 660
\pinlabel{(v)} at 105 10
\pinlabel{(vi)} at 320 10
\pinlabel{(vii)} at 520 10
\pinlabel{(viii)} at 725 10
\endlabellist
\includegraphics[width=126mm]{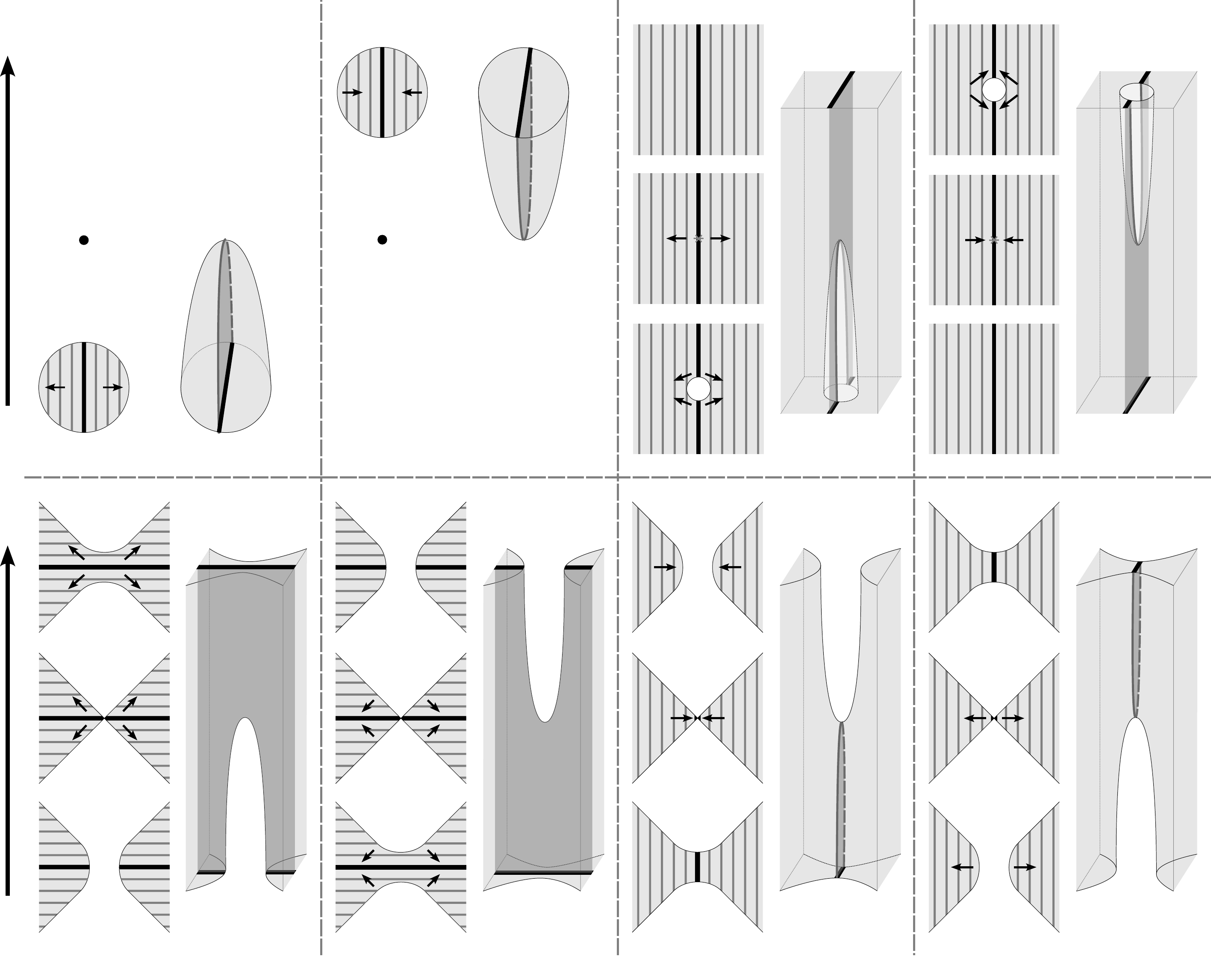}
\caption{Local models near singular leaves of $f_t$ that are also singular points of $h^{-1}(t)$. In each cell, we highlight cross-sections of $F^{-1}(\theta)$ as $t$ increases and give a schematic of $F^{-1}(\theta)$.In each of these models, $\frac{d}{dt}f_t$ is zero at the singular point of $h^{-1}(t)$. We also include arrows in nonsingular cross-sections as in Figure \ref{fig:arrows}. These arrows are drawn assuming that the non-highlighted leaves of $F$ are also locally nonsingular. We draw each singularity as a boundary singularity; doubling these figures yields corresponding local models of singular leaves of $f_t$ with interior singular points.}\label{fig:singulararrows}
\end{figure}



\begin{proposition}\label{prop:totalfunction}
The total function $F$ of a movie of singular fibrations $f_t$ is a fibration as long as the following are true.
\begin{itemize}
\item Near singularities in leaves of $f_t$ that are not also singularities of the surface $h^{-1}(t)$, $\frac{d}{dt}f_t$ is nonvanishing. (That is, we may always draw arrows on contour sets of $f_t$ as in Figures \ref{fig:arrows} and \ref{fig:singulararrows}.
\item At singularities of $h^{-1}(t)$, $\frac{d}{dt}f_t$ vanishes and the arrows decorating $f_{t\pm\epsilon}$ are as in one of the models of Figure \ref{fig:singulararrows}.
\item The maps $f_0$ and $f_1$ are both fibrations.
\end{itemize}
\end{proposition}
\begin{proof}
Assume the listed conditions hold. This ensures that the leaves of $F$ are nonsingular, so $F:M\to S^1$ is a fibration.
\end{proof}

\begin{remark}
When the total function $F$ of $f_t$ is a fibration, then $h$ restricts to a height function on each fiber of $F$. Figures \ref{fig:arrows} and \ref{fig:singulararrows} are simply models of interior and boundary critical points of $h|_{F^{-1}(\theta)}$. 
\end{remark}

\begin{notation}
Assume $f_t$ has arrow decorations. 
A leaf of $f_t$ containing an $\times$-singularity divides a small neighborhood of that singularity into four regions. Two of these regions (opposite to each other) locally have arrows pointing into them; we call these the {\emph{inward regions}}. The other two regions locally have arrows pointing away from them; we call these the {\emph{outward regions}}. See Figure \ref{fig:xsingularity}.
\end{notation}

\begin{figure}
\centering
\labellist
\pinlabel{\rotatebox{90}{inward region}} at -7 42
\pinlabel{\rotatebox{90}{inward region}} at 88 42
\pinlabel{outward region} at 42 -6
\pinlabel{outward region} at 42 88
\endlabellist
\vspace{.1in}
\includegraphics[width=30mm]{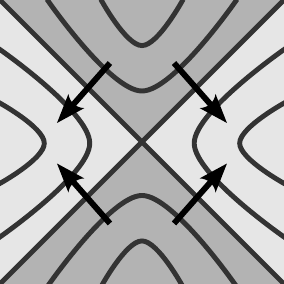}
\caption{Inward and outward regions of a neighborhood of an $\times$-singularity.}\label{fig:xsingularity}
\end{figure}
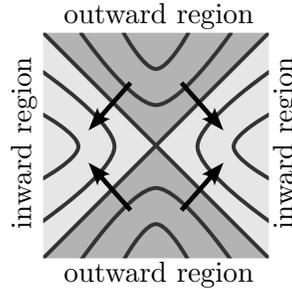



\section{Small movies}
Here we give three main examples of movies of singular fibrations that we will use to construct more interesting movies. 
To prove Theorem \ref{thm:main}, we will only need to consider singular fibrations on surfaces with only boundary singularities. From now on, we will only consider singular surfaces with no interior singularities.

\begin{example}\label{ex:min}
Parametrize a 3-ball $B^3$ as $I\times I\times I$, and let $\pi:B^3\to I$ be projection to the third factor. Fix smooth $f_0:\pi^{-1}(0)\to S^1$ so that the level sets of $f_0$ are lines of the form $0\times\{I\}\times\{0\}$.

Let $\gamma$ be an arc properly embedded in $B^3$ so that $\pi(\partial\gamma)=1$, and $\pi|_{\gamma}$ is Morse with exactly one local minimum that is contained in $\pi^{-1}(1/2)$, and projecting $\gamma$ to $I\times I\times 0$ yields an arc contained in a level set of $f_0$.

Let $M^3=B^3\setminus\nu(\gamma)$. We choose the thickening of $\gamma$ so that $h:=\pi|_M$ has two singular level sets. See Figure \ref{fig:minmovie}. In the top row of Figure \ref{fig:Lminimum} we illustrate an extension of $f_0$ to a movie of singular fibrations $f_t:h^{-1}(t)\to S^1$. This is a concatenation of movies (iv) and (viii) of Figure \ref{fig:singulararrows}, so the arrows are consistent with the requirements of Proposition \ref{prop:totalfunction}. 
The leaves of the resulting total function $F:M\to S^1$ are nonsingular in their interior, but there are two $\times$-singularities in leaves of $f_1$.
\end{example}

\begin{figure}
\centering
\labellist
\pinlabel{$t$} at 313 173
\pinlabel{$t$} at 313 21
\endlabellist
\includegraphics[width=126mm]{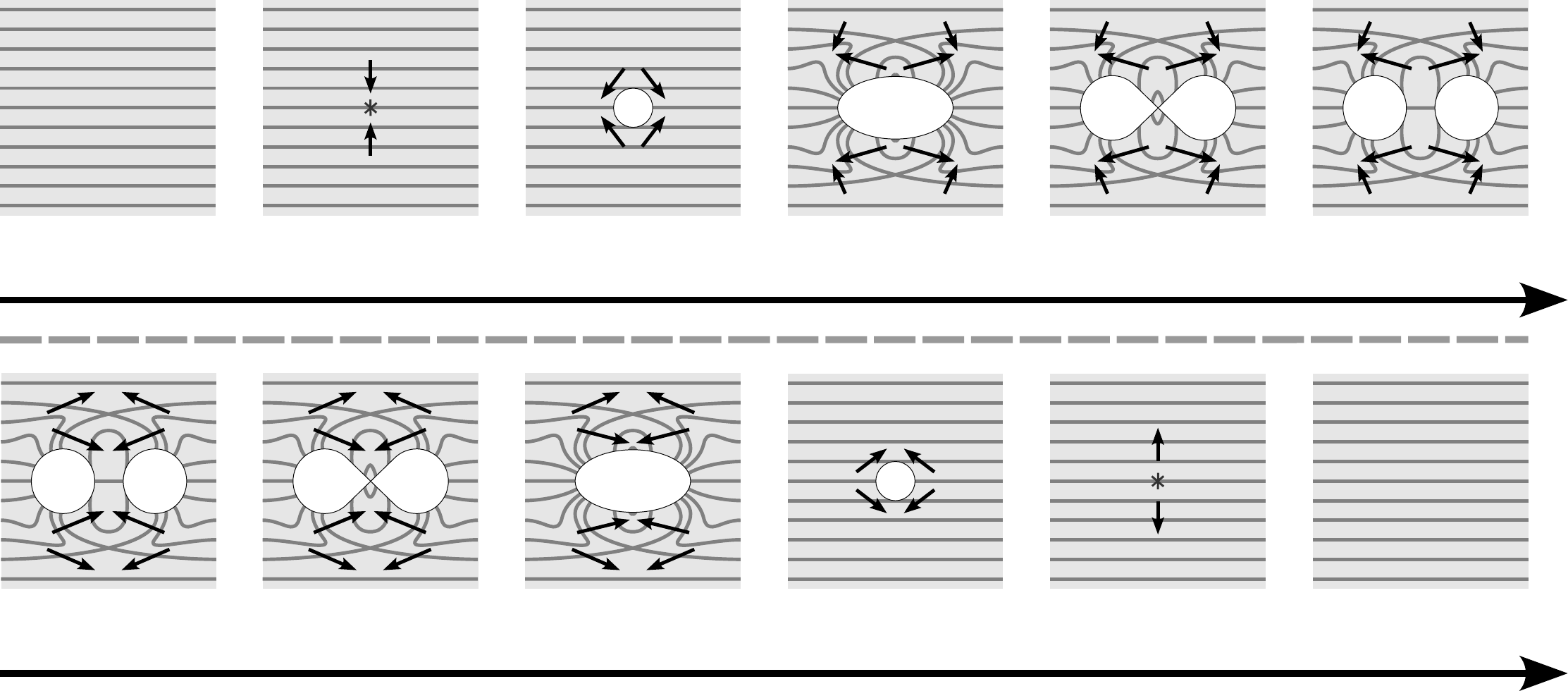}
\caption{{\bf{Top row:}} the movie of Example \ref{ex:min}. {\bf{Bottom row:}} the movie of Example \ref{ex:max}.}\label{fig:Lminimum}
\end{figure}

\begin{figure}
\includegraphics[width=75mm]{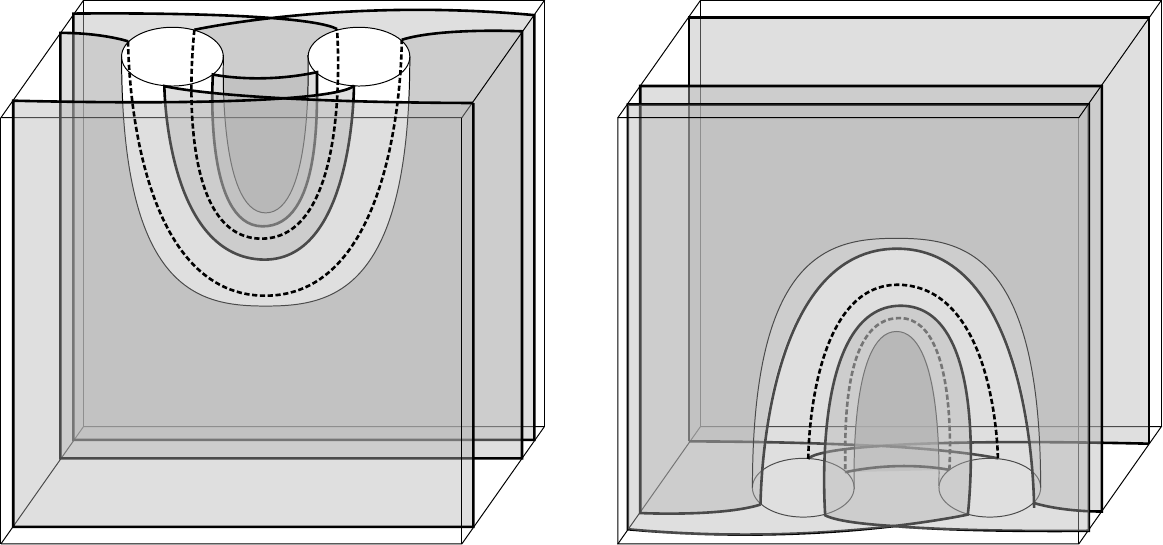}
\caption{Some leaves of the total function of the movies described in Example \ref{ex:min} (left) and Example \ref{ex:max} (right).}\label{fig:minmovie}
\end{figure}

When attempting to fiber the complement of a knot or link $L$, we can use the movie in Example \ref{ex:min} as a model of a fibration near a minimum of $L$.

\begin{example}\label{ex:max}
Now we will essentially turn Example \ref{ex:min} upside down. Again parametrize a 3-ball $B^3$ as $I\times I\times I$, and let $\pi:B^3\to I$ be projection to the third factor.  Let $f_0:h^{-1}(0)\to S^1$ agree with the function $f_1$ of Example \ref{ex:min}. Fix an arc $\gamma$ properly embedded in $B^3$ so that $\pi(\partial\gamma)=0$, and $\pi|_{\gamma}$ is Morse with exactly one local maximum that is contained in $\pi^{-1}(1/2)$, and projecting $\gamma$ to $I\times I\times 0$ yields an arc contained in a level set of $f_0$. 
In the bottom row of Figure \ref{fig:Lminimum} we illustrate an extension of $f_0$ to a movie of singular fibrations $f_t:h^{-1}(t)\to S^1$. This is a concatenation of movies (vii) and (iii) of Figure \ref{fig:singulararrows}, so the arrows are consistent with the requirements of of Proposition \ref{prop:totalfunction}. The leaves of the resulting total function $F:M\to S^1$ are nonsingular in their interior, but there are two $\times$-singularities in leaves of $f_0$.
\end{example}

\begin{figure}
\centering
\labellist
\pinlabel{$t=0$} at 31 17
\pinlabel{$t=1$} at 342 17
\endlabellist
\includegraphics[width=100mm]{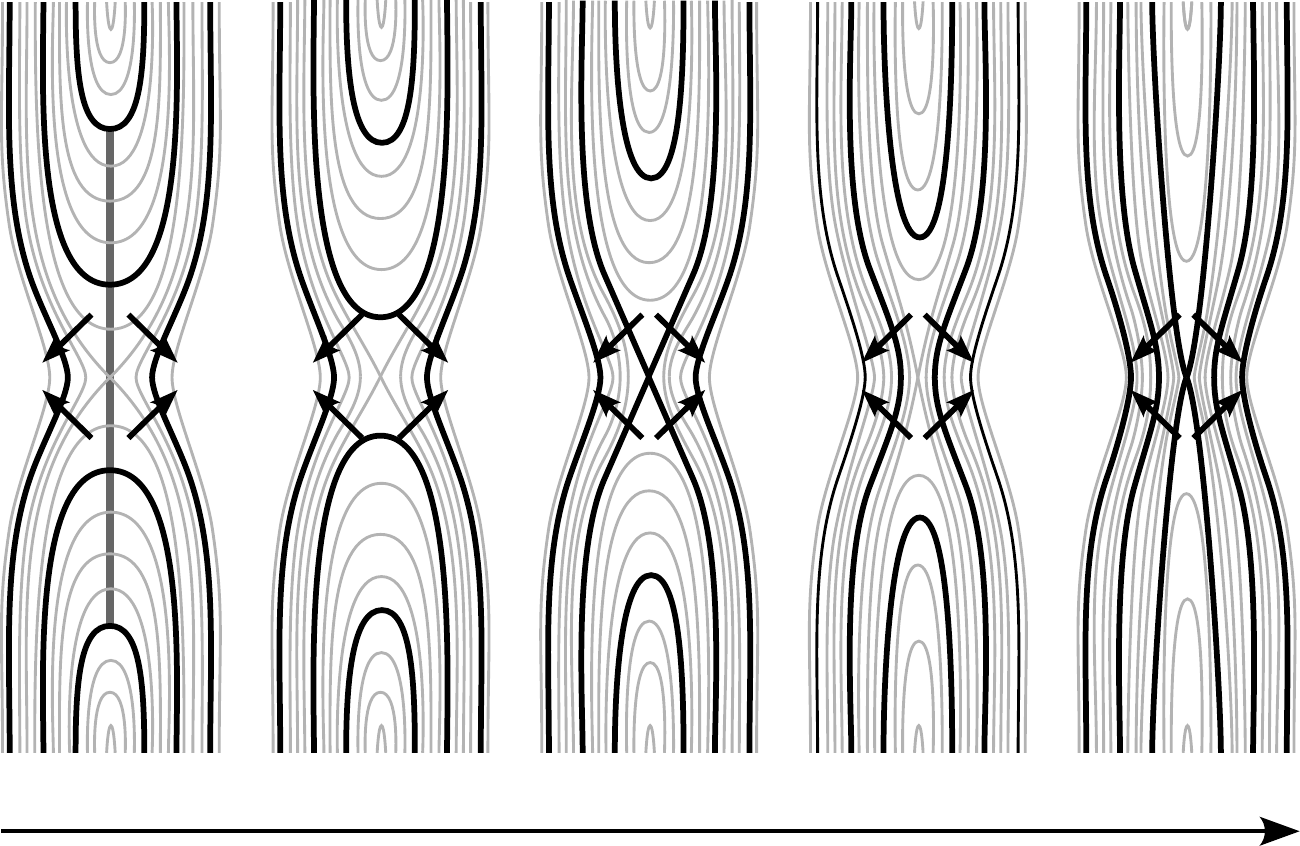}
\caption{The movie of Example \ref{ex:band} (band movie). {\bf{Leftmost frame:}} A neighborhood of the arc $\eta$ (bold vertical arc) in the domain of $f_0$. The $\times$-singularity $p$ of $f_0$ lies at the center of $\eta$. Also in bold, we draw the leaf $f_0^{-1}(\theta)$ that contains the endpoints of $\eta$. The leaf $f_0^{-1}(\theta)$ may intersect $\eta$ in more points as shown (or even include the point $p$). The circular length condition on $\eta$ ensures that $\eta$ intersects $f_0^{-1}(\theta)$ the same number of times on either side of $p$. {\bf{From left to right:}} $f_t$ for increasing $t$.
}\label{fig:bandmovie}
\end{figure}

When building a fibration of a knot or link in braid position, we will use the movie in the following example as a model of the fibration near a crossing. It may be helpful to think of a braid closure as obtained by performing band surgeries on an unlink, with bands corresponding to crossings in the braid. In the following example, we show how to change $f_t^{-1}(\theta)$ by band surgery to the leaves along an arc $\eta$ (the core of a band) as $t$ increases from $0$ to $1$, assuming that $f_0$ is sufficiently nice. Since we will use this model several times, we will refer to the movie of Example \ref{ex:band} as a {\emph{band movie}}. We suggest that the reader consider Figure \ref{fig:bandmovie} before reading the text of Example \ref{ex:band}: we draw an arc $\eta$ connecting one leaf of $f_0$ to itself, such that the arc crosses over exactly one $\times$-singularity, meeting the outward regions. We then build a movie of singular fibrations in which the arc is contracted, so that in $f_1$, each leaf of $f_0$ that met $\eta$ has been changed by some number of band surgeries.

\begin{example}\label{ex:band}
Parametrize a 3-ball $M^3$ as $I\times I\times I$, and let $h:M^3\to I$ be projection to the third factor.
Let $f_0:\pi^{-1}(0)\to S^1$ be smooth so that $f$ has one $\times$-singularity at a point $p$. 
Fix arrow decorations at $p$. Let $\eta$ be an arc in $I\times I=h^{-1}(0)$ with $\eta(1/2)=p$, meeting the outward regions near $p$. (Recall arrows near $\times$-singularities point out of outward regions and into inward regions.) Take $\eta$ to be transverse to all leaves of $f_0$ away from $p$. Assume $f_0(\eta(0))=f_0(\eta(1))=\theta$ for some $\theta$ and that in the intervals $[0,1/2]$ and $[1/2,1]$, $f_0\circ\eta$ winds in equal magnitude about $S^1$. That is, assume there is an $l\in\mathbb{R}$ with 
\[l=\int_{0}^{1/2}f_0(\eta(t))dt=-\int_{1/2}^{1}f_0(\eta(t))dt.\]

Let $b:I\times I\to I$ be a bump function supported on a small neighborhood of $\eta$, with $b(\eta)=1$. Then set
\[f_t(x)=f_0(x)-t\cdot l\cdot b(x).\] We draw $f_t$ in Figure \ref{fig:bandmovie}. Note that $f_1(p)=f_0(p)-l=\theta$. Our requirement that $\eta$ pass through the outward regions near $p$ ensures that the decoration arrows on $f_0$ actually correspond to the movie $f_t$. 

Since this movie satisfies the first two conditions of Proposition \ref{prop:totalfunction}, the total function $F$ has leaves that are nonsingular away from $h^{-1}(\{0,1\})$.
\end{example}

\section{Proof of Theorem \ref{thm:main}}

Finally, we will re-prove the result of Stallings \cite{stallings} by explicitly constructing a fibration on the complement of any homogenous braid closure. From now on, let $L$ be a link in braid position as the closure of a homogeneous braid $\beta$ on $b$ strands. Let $w=w_1\cdots w_n$ be a homogenous braid word for $\beta$, so $w_1,\ldots, w_n\in\{\sigma_1^{s_1},$ $\ldots,$ $\sigma_{b-1}^{s_{b-1}}\}$ for some $s_1,\ldots, s_{b-1}\in\{-1,1\}$, with each $\sigma_i^{s_i}$ appearing at least once in $w$.

Since $L$ is in braid position we have a decomposition $S^3=B^3_+\cup (S^2\times I) \cup B^3_-$, where $L$ meets $S^2\times I$ in $\beta$ and a trivial $b$-stranded braid, while $L$ meets each of $B^3_+$ and $B^3_-$ in a trivial $b$-stranded tangle. The unique minimum of the standard Morse function $h:S^3\to\mathbb{R}$ is in $B^3_-$ and is also a minimum of $L$, while the unique maximum is in $B^3_+$ and is a maximum of $L$. 

Isotope as necessary so that the minima of $L$ are in ascending order according to the projection of $\beta$ described by $w$. That is, from bottom to top, the $i$-th minimum of $L$ lies below the $i$-th endpoint of $\beta\cap\partial B^3_-$. 
Similarly isotope the maxima to be in descending order, so from top to bottom the $i$-th maximum of $L$ is in the strand above the $i$-th endpoint of $\beta\cap\partial B^3_+$. (This is illustrated in Figure \ref{fig:stepschematic}.)


Reparametrize $h$ so that $h^{-1}((-\infty,0])$ and $h^{-1}([3,\infty))$ are both balls meeting $L$ in a single arc, and so that 
$h^{-1}(1)=\partial B^3_+$ while $h^{-1}(2)=\partial B^3_+$. (Again, see Figure \ref{fig:stepschematic}.)

We will build a movie of singular fibrations on $M^3:=S^3\setminus\nu(L)$ with respect to $h|_M$ whose total function has nonsingular leaves as in Proposition \ref{prop:totalfunction} and satisfies the conditions of Theorem \ref{thm:main}, thus proving Theorem \ref{thm:main}. In Figure \ref{fig:stepschematic}, we give a schematic of the position of $L$ and which portions of $M$ are fibered in each step of the construction.


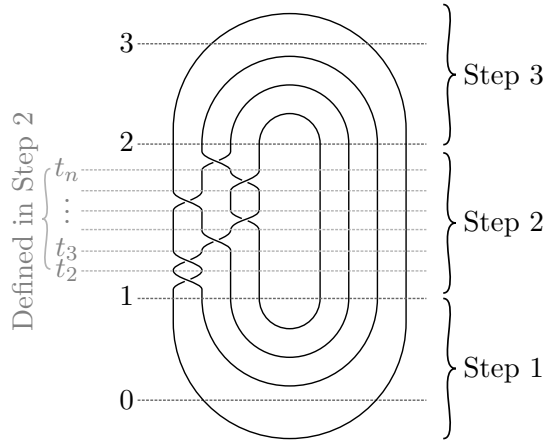
\begin{figure}
\centering
\labellist
\pinlabel{$0$} at 85 40
\pinlabel{$1$} at 85 145
\pinlabel{$2$} at 85 300
\pinlabel{$3$} at 85 400
\pinlabel{\textcolor{gray}{$t_2$}} at 25 171
\pinlabel{\textcolor{gray}{$t_3$}} at 25 192
\pinlabel{\textcolor{gray}{$\vdots$}} at 25 238
\pinlabel{\textcolor{gray}{$t_n$}} at 25 272
\pinlabel{\rotatebox{90}{\textcolor{gray}{Defined in Step 2}}} at -17 220
\pinlabel{Step 1} at 465 70
\pinlabel{Step 2} at 465 217
\pinlabel{Step 3} at 465 366
\endlabellist
\hspace{-.7in}\includegraphics[width=55mm]{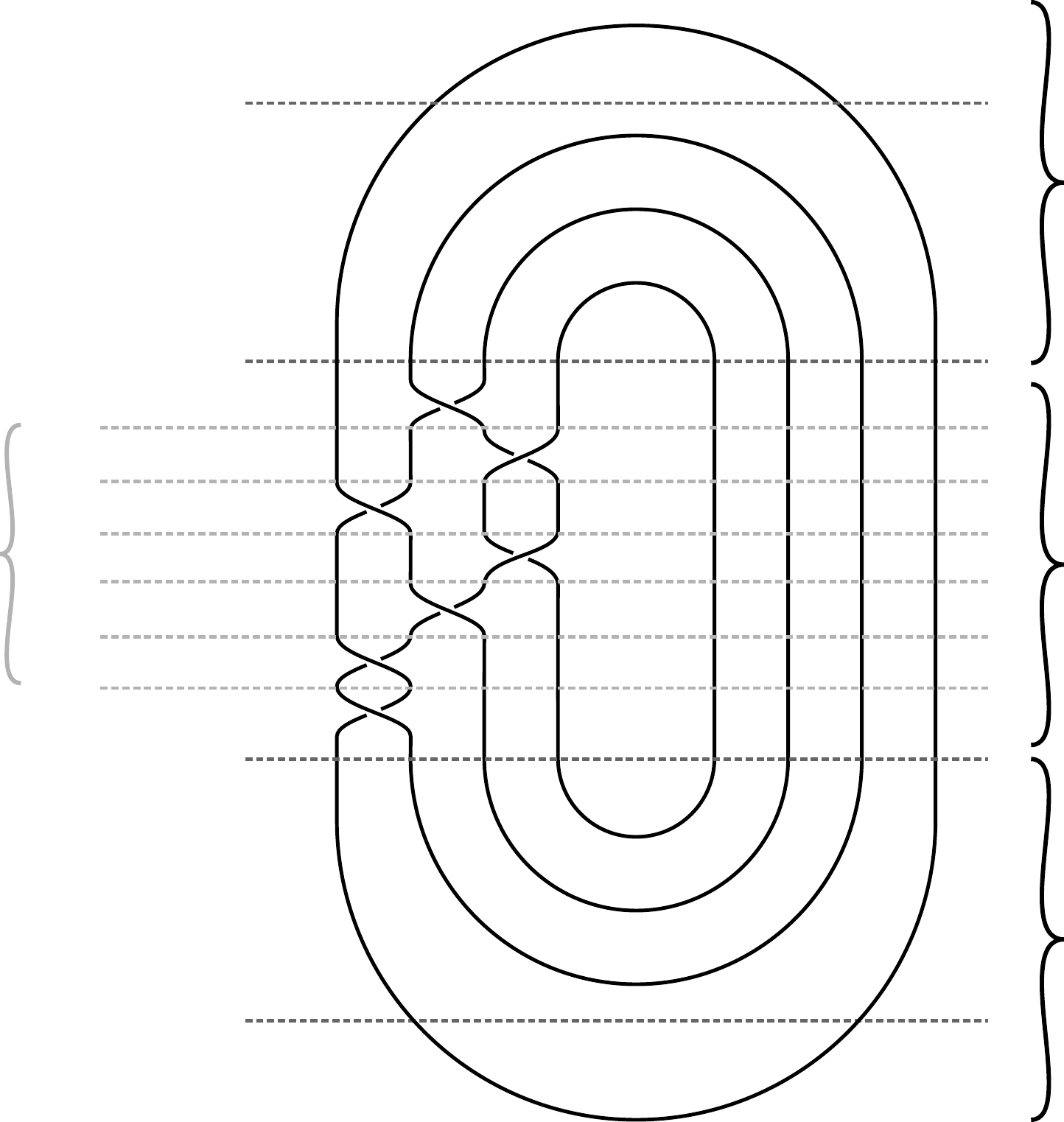}
\caption{A schematic of the values of $h|_M$ relevant to each step of the construction in Theorem \ref{thm:main}.}\label{fig:stepschematic}
\end{figure}

\subsection{Step 1: fibering $B^3_-\setminus\nu(L)$}


Since $h|_M^{-1}((-\infty,0])$ is a solid torus, it can be fibered by meridional disks. Define $f_t$ for $t\le0$ so that the total function from $-\infty<t\le 0$ is this meridional disk fibration, as in Figure \ref{fig:step1}. We call the two boundary components of $h|_M^{-1}(0)$ by the names $C_1$ and $C_1'$.

\begin{figure}
\centering
\labellist
\pinlabel{$t$} at 198 12
\pinlabel{$t=0$} at 362 12
\endlabellist
\includegraphics[width=110mm]{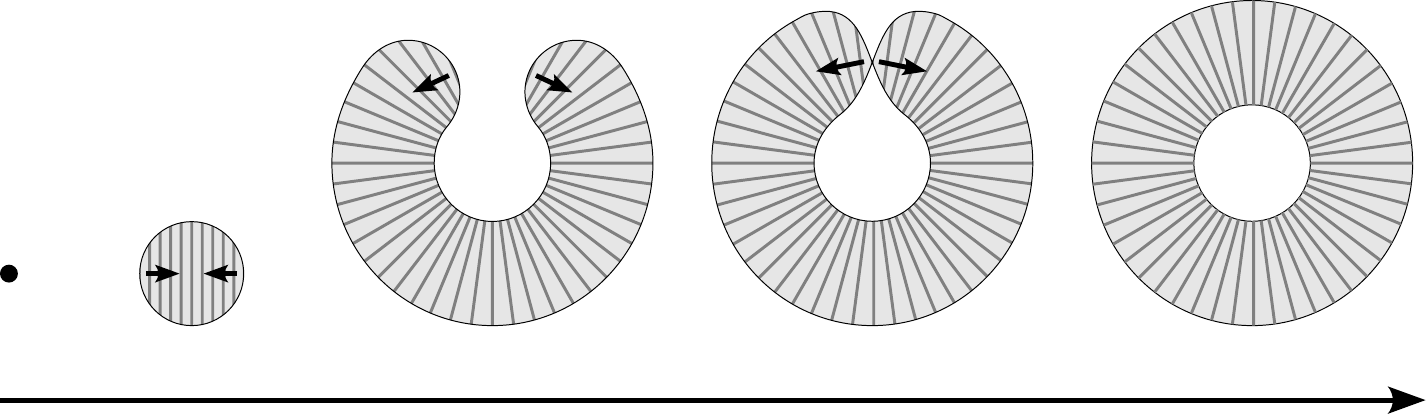}
\caption{The movie $f_t$ for $t\le 0$. The domain of $f_1$ is the annulus $h|_M^{-1}(1)\setminus\nu(L)$, i.e.\ a 2-sphere with neighborhoods of two points in $L$ deleted. Every leaf of $f_0$ is an arc between these two points (restricted to $M$). This movie is a concatenation of movies (ii) and (viii) of Figure \ref{fig:singulararrows}.}
\label{fig:step1}
\end{figure}

Now we extend $f_t$ across all of $B^3_-\setminus\nu(L)$. The reader may find it helpful to begin by consulting Figure \ref{fig:step1moreminima}, in which we show how to adapt a singular fibration of a planar surface (in this case, some $h^{-1}(t)$ slightly below a local minimum of $h|_L$) to a planar surface with two additional boundary components (in this case, $h^{-1}(t')$ slightly above a local minimum of $h|_L$).

Note $L\cap B^3_-$ has $b$ components, which we call $A_1,\ldots, A_{b}$ in ascending order. The arc $A_1$ is the only arc that meets $h^{-1}(0)$. From $t=0$ to $t=\epsilon$, we extend $f_t$ via isotopy, rotating an arc beneath $A_2$ by $180^\circ$ as depicted in Figure \ref{fig:step1moreminima}. The sign of this rotation depends on $s_{1}$. If $s_{1}=1$ (i.e.\ $\sigma_{1}$ appears in $w$), then we take the rotation to be counterlockwise, as in the top row of Figure \ref{fig:step1moreminima}. If $s_1=-1$ (i.e.\ $\sigma_1^{-1}$ appears in $w$), then we take the rotation to be clockwise, as in the bottom row of Figure \ref{fig:step1moreminima}. 

\begin{figure}
\centering
\labellist
\pinlabel{$t$} at 275 161
\pinlabel{$t$} at 275 13
\pinlabel{$s_{i-1}=1$} at 275 274
\pinlabel{$s_{i-1}=-1$} at 275 127
\endlabellist
\includegraphics[width=126mm]{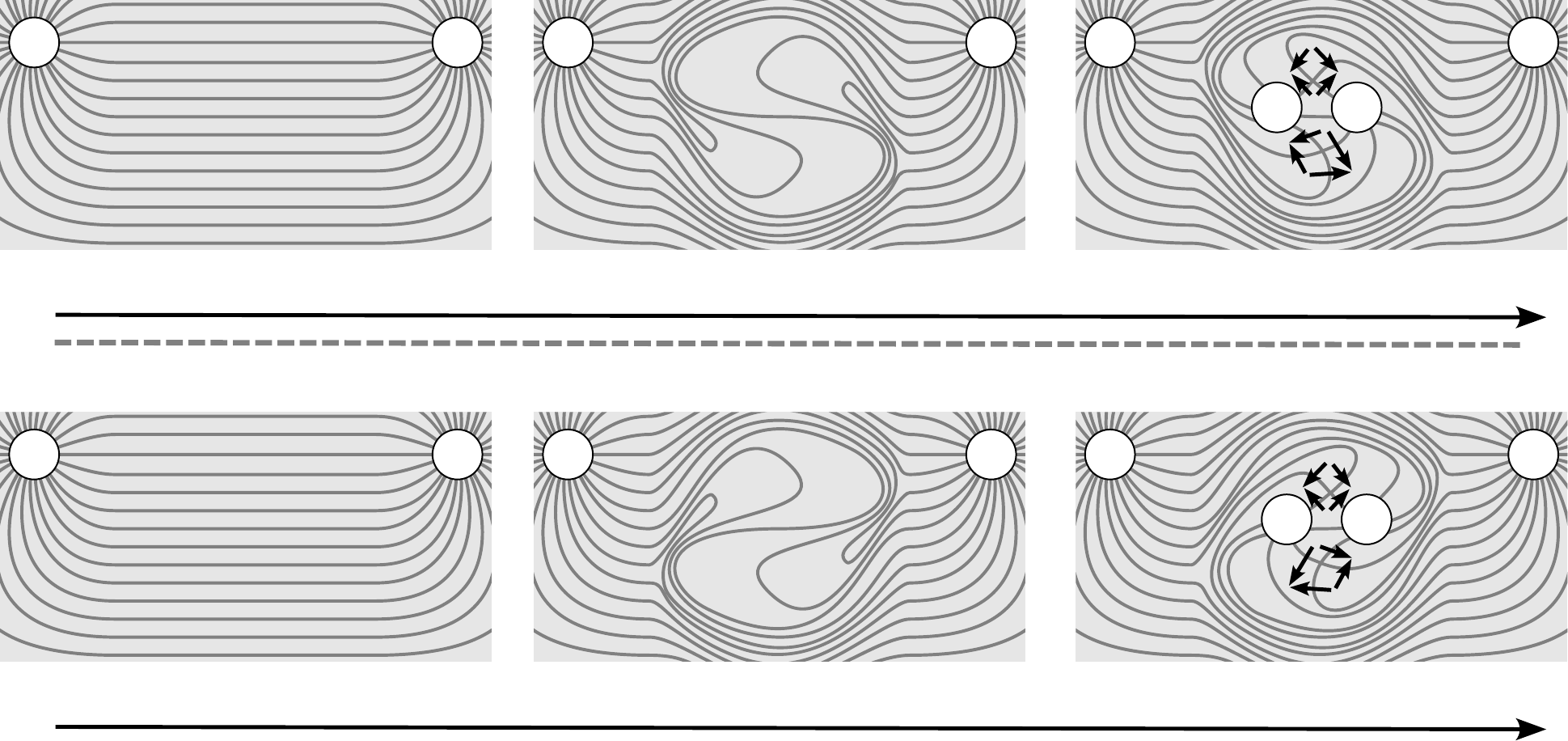}
\caption{An illustration of Step 1. In both rows, we obtain $f_t$ by isotoping $f_0$ and extending above the minimum of $A_i$ via as band movie (as in Example \ref{ex:band}). {\bf{Top row:}} $s_{i-1}=1$, so we rotate counterclockwise. {\bf{Bottom row:}} $s_{i-1}=-1$, so we rotate clockwise.}\label{fig:step1moreminima}
\end{figure}
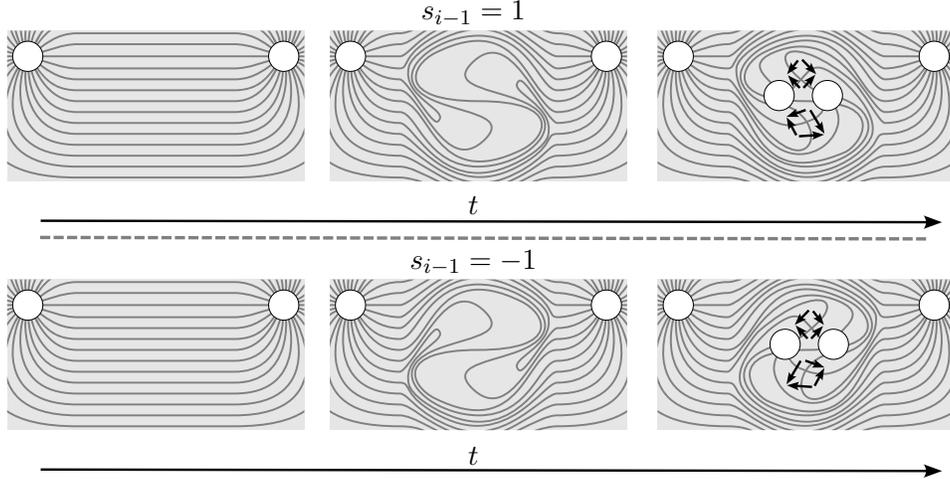

We now extend $f_t$ above the minimum of $A_2$ via the movie of  
Example \ref{ex:min}. This adds two new boundary components to the domain of $f_t$ (for the greatest value of $t$ for which $f_t$ is currently defined), which we call $C_2$ and $C_2'$.

Now we can explain why we first performed a $180^\circ$ rotation. Roughly, we want all boundary components of $h^{-1}(t)$ in the ``left" half of the page to correspond to meridians of strands in the braid $\beta$, and the boundary components in the ``right" half of the page to correspond to vertical strands, as suggested by Figure \ref{fig:stepschematic}. The $180^\circ$ rotation causes the ``left" boundaries in Figure \ref{fig:step1moreminima} to have one sign and the ``right" boundaries in Figure \ref{fig:step1moreminima} to have the opposite. The purpose of having the direction of rotation determined by $s_1$ will be made clear in the next step.

Repeat for $A_i$, $i=3,\ldots, b$ (using a rotation of sign $s_{i-1}$) so that $f_t$ is defined for all $t\le 1$ and $h_M^{-1}(1)$ has boundary components $C_1,C_1',\ldots,C_b,C_b'$. We illustrate the leaves of $f_1$ and arrow decorations near the $2(b-1)$ resulting $\times$-singularities in the left half of Figure \ref{fig:step1allminima}.

\begin{figure}
\centering
\labellist
\small
\pinlabel{$C_1$} at 180 420
\pinlabel{$C_1'$} at 242 420
\pinlabel{$C_b$} at 180 165
\pinlabel{$C_b'$} at 242 165
\pinlabel{$C_1$} at 642 420
\pinlabel{$C_1'$} at 580 420
\pinlabel{$C_b$} at 642 165
\pinlabel{$C_b'$} at 550 165
\endlabellist
\includegraphics[width=126mm]{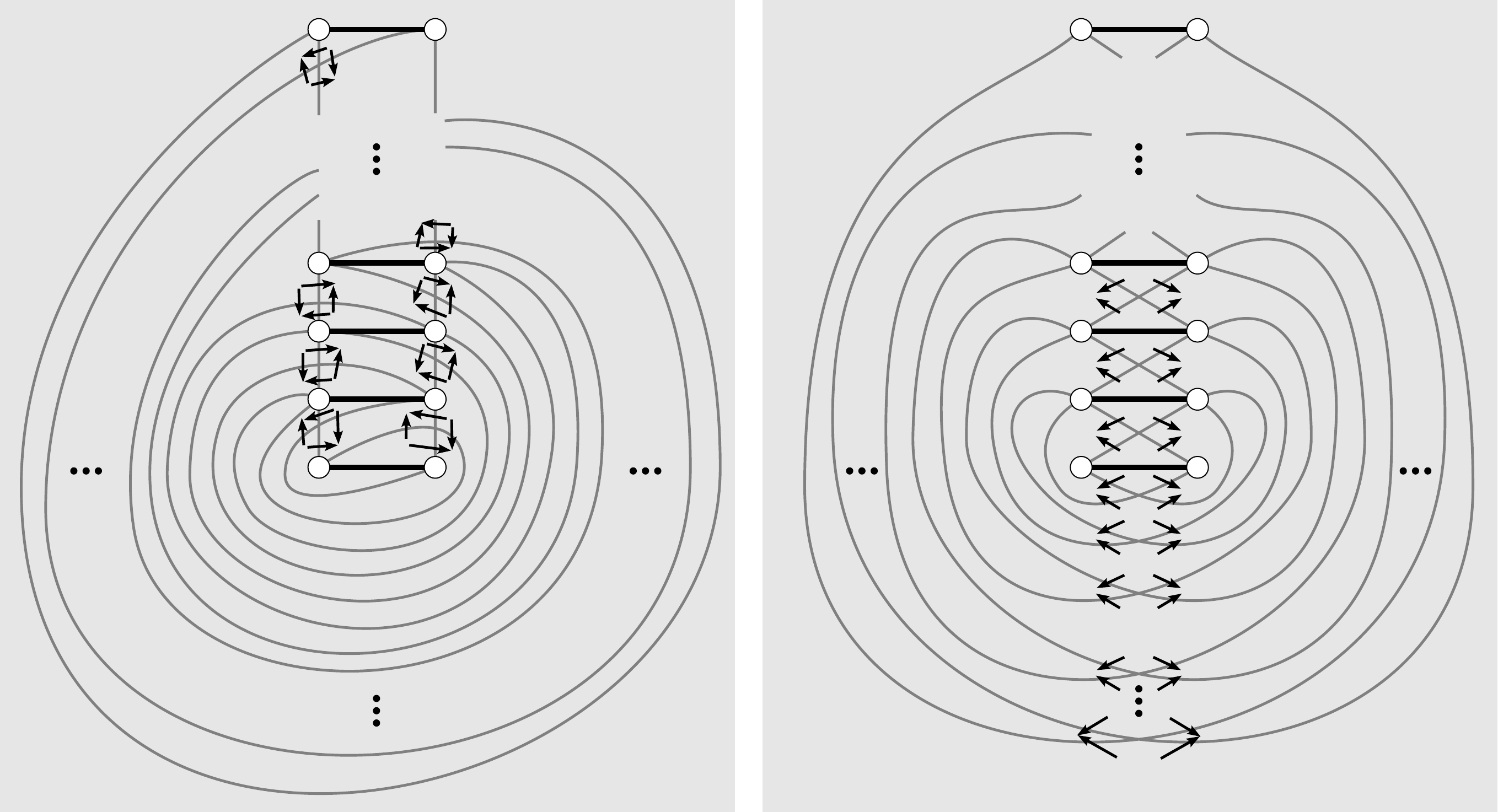}
\caption{{\bf{Left:}} A contour map of $f_1$ as obtained in Step 1. From top to bottom, the horizontal pairs of boundary components are near $A_1,\ldots, A_b$. In this figure, we have assumed $s_1=1$, $s_{b-4}=1$, $s_{b-3}=-1$, $s_{b-2}=-1$, $s_{b-1}=1$. The left column of boundary components from top to bottom are $C_1,\ldots, C_b$, while the right column of boundary components are from top to bottom $C_1',\ldots, C_b'$. In bold, we draw the projections of the arcs $L\cap B^3_-$ to $h|_M^{-1}(1)$. Each of these arcs is in a level set of $f_1$. {\bf{Right:}} Another contour map obtained by isotoping the lefthand diagram. This isotopy rotates each pair $C_i, C_i'$ ($i=2,\ldots, b$) through $180^{\circ}$ in the direction determined by $s_i$, exchanging the two boundary components.}\label{fig:step1allminima}
\end{figure}
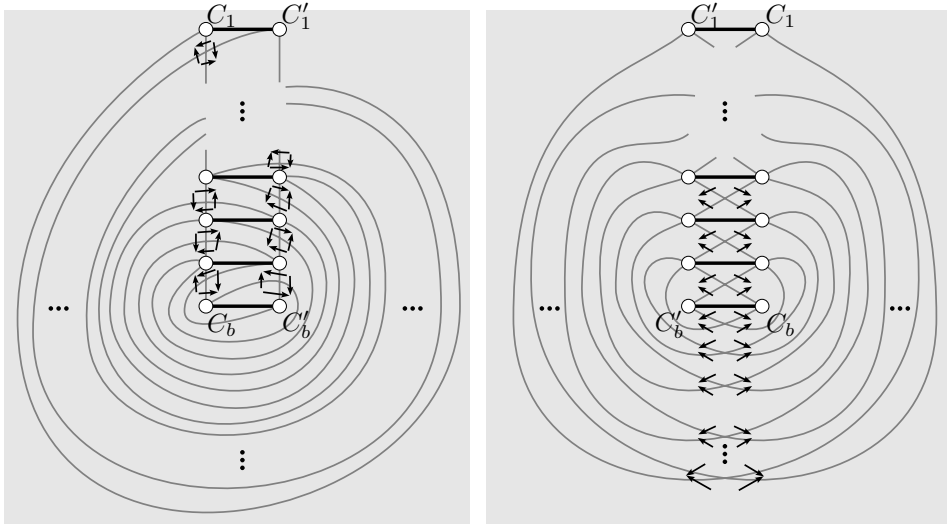

In the right half of Figure \ref{fig:step1allminima} we draw an alternative view of the leaves of $f_1$; the first view will be more useful but it might be easier for a reader to see that the second arises from applying $b-1$ instances of the movie of Example \ref{ex:min} to $f_0$. The two images of Figure \ref{fig:step1allminima} are related by an isotopy that rotates horizontal pairs of boundary components through 180$^\circ$, exchanging the two boundaries.


\subsection{Step 2: fibering $h|_M^{-1}[1,2]$}
Let $1=t_1<t_2<\cdots<t_n<t_{n+1}=2$ so that $\beta$ meets $h^{-1}[t_i,t_{i+1}]$ in a braid described by the $i$-th letter $w_i$ of $w$. Our goal is to extend $f_t$ across each $[t_i,t_{i+1}]$ by viewing the $i$-th crossing of $\beta$ as a band move and applying the band movie of Example \ref{ex:band}.

Say $w_1=\sigma_j^{s_j}$. In the left of Figure \ref{fig:sigmanew}, we illustrate an open set of $h_M^{-1}(1)$ that has boundary circles $C_{j-1}$, $C_{j-1}$, $C_j$, $C_j'$, $C_{j+1}$, and $C_{j+1}'$. (If $j=1$, then ignore the top portion including $C_{j-1}, C_{j-1}$.) We assume $s_j=1$; the $s_j=-1$ case is similar (mirror along a horizontal axis and take the included boundary components to be $C_j,C_j',C_{j+1},C_{j+1}',C_{j+2},C_{j+2}'$). This subspace is easily visible in the left of Figure \ref{fig:step1allminima}. We give two cases, depending on the sign of $s_{j-1}$. 
In Figure \ref{fig:sigmanew}, we illustrate how to extend $f_t$ across $t\in[1,t_2]$. We first perform a band movie along the illustrated arc $\eta$. (One can check that the endpoints of $\eta$ can be taken to have the same image under $f_1$, as required.) Then $C_j,C_{j+1}$ exchange positions as $t$ increases, after which we exchange their labels. We see that the contour set of $f_{t_2}$ agrees with $f_1$ away from a disk containing $C_j,C_j',C_{j+1},C_{j+1}'$.

\begin{figure}
\centering
\labellist
\pinlabel{$t_i$} at 51 201
\pinlabel{$t_i$} at 51 11
\pinlabel{$t_{i+1}$} at 300 201
\pinlabel{$t_{i+1}$} at 300 11
\pinlabel{$j-1$} at -16 335
\pinlabel{$j$} at -16 287
\pinlabel{$j+1$} at -16 235
\pinlabel{$j-1$} at -16 145
\pinlabel{$j$} at -16 97
\pinlabel{$j+1$} at -16 45
\endlabellist
\includegraphics[width=90mm]{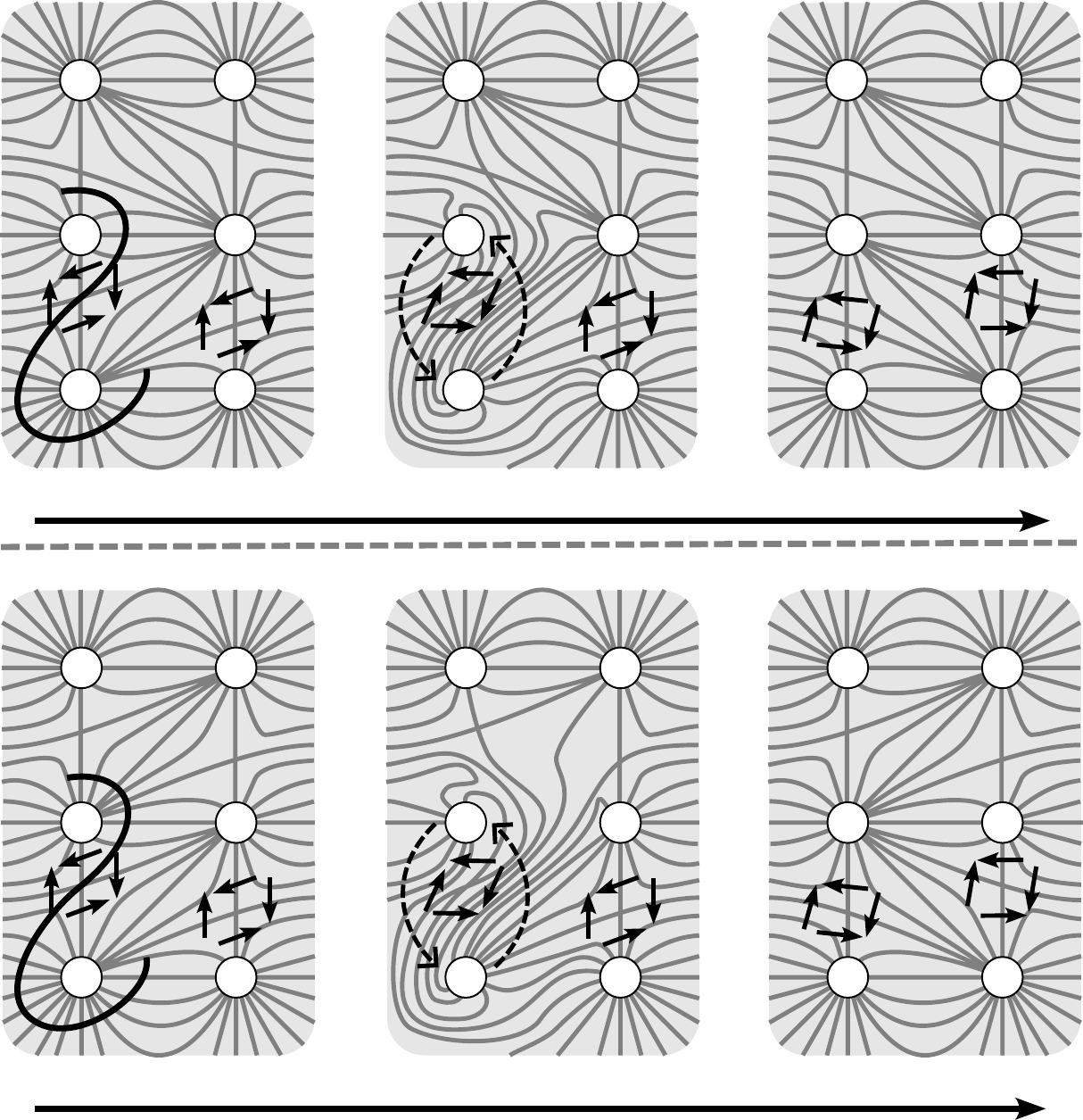}
\caption{Here we assume $s_j=1$.  We give two possible local models depending on the behavior of $f_{t_i}$ in a box containing $C_{j-1},C_{j-1}',C_j,C_j'$.  This depends on the sign of $s_{j-1}$ and whether we have already extended $f_t$ across some interval corresponding to $\sigma_{j-1}^{s_{j-1}}$. {\bf{Left:}} An open subset of $h_M^{-1}(t_i)$. We include an arc $\eta$ satisfying the conditions of the band movie. {\bf{Middle:}} We perform a band movie and indicate with dashed arrows the path of $C_j,C_{j+1}$ as we move above the crossing of $\beta$ corresponding to $w_i$. {\bf{Right:}} An open subset of $h_M^{-1}(t_{i+1})$.}\label{fig:sigmanew}
\end{figure}
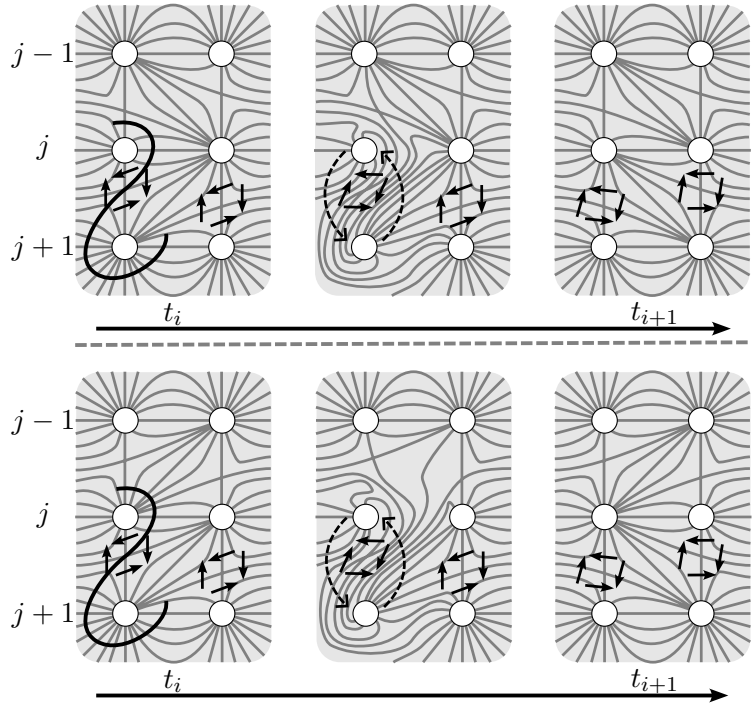


We similarly extend $f_t$ along each interval $t\in[t_i,t_{i+1}]$ unless the letter $w_i$ has already appeared in $w$. In this case, suppose $w_i=\sigma_j^{s_j}$ and $w_k=\sigma_j^{s_j}$ for some $k<i$. Then we must have performed the above operation in a neighborhood of $C_j,C_j',C_{j+1},C_{j+1}'$ when extending $f_t$ along $t\in[t_k,t_{k+1}]$. Therefore, the leaves of $f_{t_i}$ have the local model of Figure \ref{fig:sigmanew} (right), again assuming $s_j=1$ (the case $s_j=-1$ is similar). In Figure \ref{fig:sigmanewsecondtime}, we show how to extend $f_t$ to $t\in[t_i, t_{i+1}]$. We perform the indicated band move and obtain a singular fibration that agrees with the middle stage of Figure \ref{fig:sigmanew}, and then proceed as in Figure \ref{fig:sigmanew}.

\begin{figure}
\centering
\labellist
\pinlabel{$t_i$} at 51 202
\pinlabel{$t_i$} at 51 12
\pinlabel{$t_{i+1}$} at 300 202
\pinlabel{$t_{i+1}$} at 300 12
\pinlabel{$j-1$} at -16 336
\pinlabel{$j$} at -16 288
\pinlabel{$j+1$} at -16 236
\pinlabel{$j-1$} at -16 146
\pinlabel{$j$} at -16 98
\pinlabel{$j+1$} at -16 46
\endlabellist
\includegraphics[width=90mm]{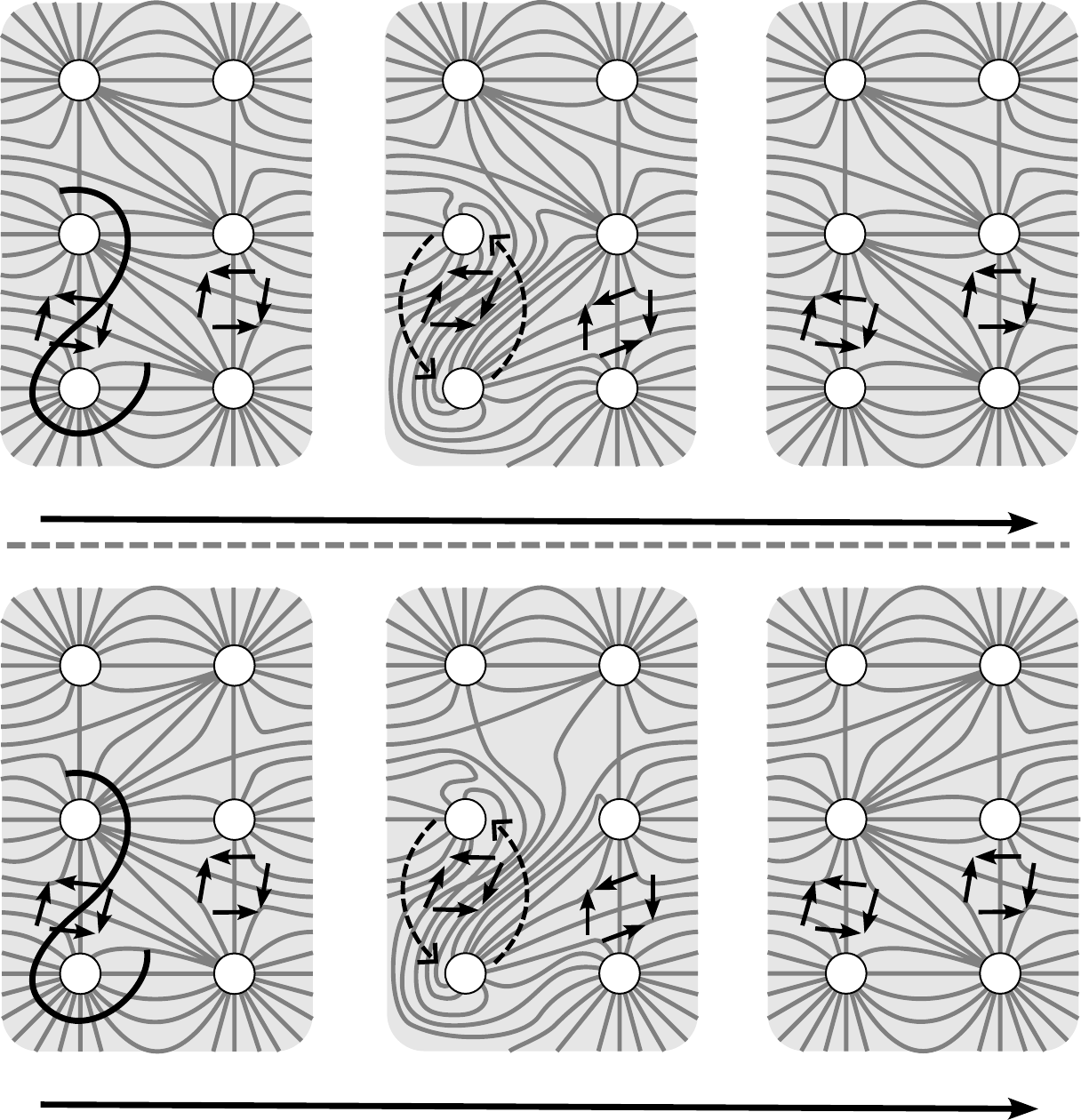}
\caption{Here, $s_j=1$ and $w_i=\sigma_j^{s_j}$ is not the first instance of $\sigma_j^{s_j}$ in $w$. We give two cases (top row and bottom row), depending on the behavior of $f_{t_i}$ in a disk containing $C_{j-1},C_{j-1}',C_j,C_j'$. (That is, depending on the sign of $s_{j-1}$ and whether $\sigma^{s_{j-1}}$ appears in the first $i-1$ letters of $w$.) {\bf{Left:}} a local model of $f_{t_i}$. We indicate an arc along which we may perform a band movie. {\bf{Middle:}} We perform the band movie. Note that this singular fibration agrees with that of the middle of \ref{fig:sigmanew}. {\bf{Right:}} An open subset of $h_M^{-1}(t_{i+1})$.
}\label{fig:sigmanewsecondtime}
\end{figure}

Thus, we may inductively extend $f_t$ over $t\in[t_i,t_{i+1}]$ for increasing $i$. If $w_i$ is the first instance of $\sigma_j^{s_j}$ in $w$, then we use the movie of Figure \ref{fig:sigmanew} (mirrored along a horizontal axis if $s_j=-1)$. If $w_i$ is not the first instance of $\sigma_j^{s_j}$, then we use the movie illustrated in Figure \ref{fig:sigmanewsecondtime} (again mirrored if $s_j=-1$).

Because every $\sigma_j^{s_j}$ appears at least once in $w$, the singular fibration $f_2$ is as in Figure \ref{fig:step3allmaxima}.
\begin{figure}
\centering
\centering
\labellist
\small
\pinlabel{$C_1$} at 172 422
\pinlabel{$C_1'$} at 234 422
\pinlabel{$C_b$} at 172 165
\pinlabel{$C_b'$} at 234 165
\pinlabel{$C_1$} at 628 422
\pinlabel{$C_1'$} at 570 422
\pinlabel{$C_b$} at 634 165
\pinlabel{$C_b'$} at 532 165
\endlabellist
\includegraphics[width=126mm]{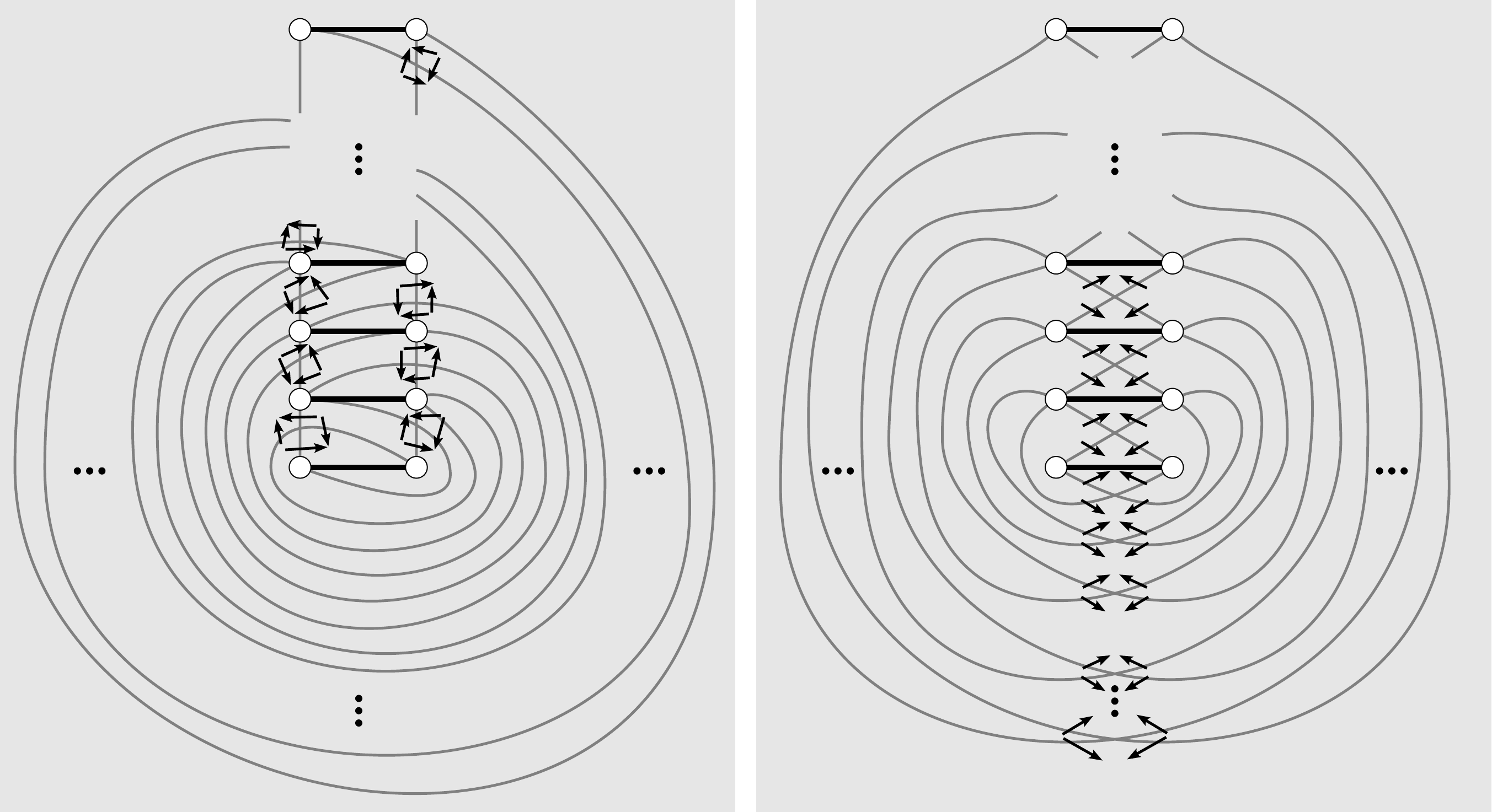}
\caption{{\bf{Left:}} A contour map of $f_2$ as obtained in Step 2. From top to bottom, the horizontal pairs of boundary components are $C_1,C_1';\ldots, C_b, C_b'$. As in Figure \ref{fig:step1allminima}, in this figure we have assumed $s_1=1$, $s_{b-4}=1$, $s_{b-3}=-1$, $s_{b-2}=-1$, $s_{b-1}=1$. In bold, we draw the projections of the arcs $L\cap B^3_+$ to $h|_M^{-1}(2)$.  {\bf{Right:}} Another contour map obtained by isotoping the lefthand diagram, included only for visualization purposes. This isotopy rotates each pair $C_i, C_i'$ $(i=2,\ldots, b$) through $180^\circ$ in the direction determined by $s_i$, exchanging the two boundary components.
}\label{fig:step3allmaxima}
\end{figure}

\subsection{Step 3: fibering $B^3_+\setminus\nu(L)$.}
We have so far extended $f_t$ to a movie of singular fibrations on $M$ for $t\le 2$.

While $f_1$ and $f_2$ (Figures \ref{fig:step1allminima} and \ref{fig:step3allmaxima}) seem similar, there is a key difference between these two singular fibrations. Projections of the arcs $L\cap B^3_-$ to $h|_M^{-1}(1)$ form a level set of $f_1$ that lies toward {\emph{outward}} regions of the $\times$-singularities, as in the movie of Example \ref{ex:min}. However, projections of arcs in $L\cap B^3_+$ to $h|_M^{-1}(2)$ form a level set of $f_2$ that lies toward {\emph{inward}} regions of the $\times$-singularities, as in the movie of Example \ref{ex:max}. We may thus extend $f_t$ across $t\in[2,3]$ by applying $b-1$ instances of the movie of Example \ref{ex:max}. (Note this is exactly the same as our procedure to extend $f_t$ along $t\in[0,1]$ from Step 1, turned upside down, so we abbreviate the discussion.) Then $f_3$ is a fibration of an annulus over $S^1$. We extend $f_t$ across $t\in[3,\infty)$ via the movie of Figure \ref{fig:step1} used in Step 1, turned upside down with respect to $t$. This concludes the construction of $f_t$ and hence the proof of Theorem \ref{thm:main}. 

\section{Example of a fibration}

In Figures \ref{fig:trefoilexample} and \ref{fig:trefoilexampleleaves}, we perform the above algorithm to construct a fibration $F:S^3\setminus($right-handed trefoil$)\to S^1$. In Figure \ref{fig:trefoilexample} we draw the contour set of each $f_t:=F|_{h^{-1}(t)}$. In Figure \ref{fig:trefoilexampleleaves} we highlight $F^{-1}(\theta)$ for three values of $\theta$. Each of these leaves can be easily seen to be the standard braid surface (as expected). Note that, as stated in Theorem \ref{thm:main}, $h$ restricts to the interior of each $F^{-1}(\theta)$ as a Morse function with no local minima or maxima. 

\begin{figure}
\centering
\labellist
\pinlabel{$t$} at -16 335
\pinlabel{$t$} at 485 335
\endlabellist
\includegraphics[width=126mm]{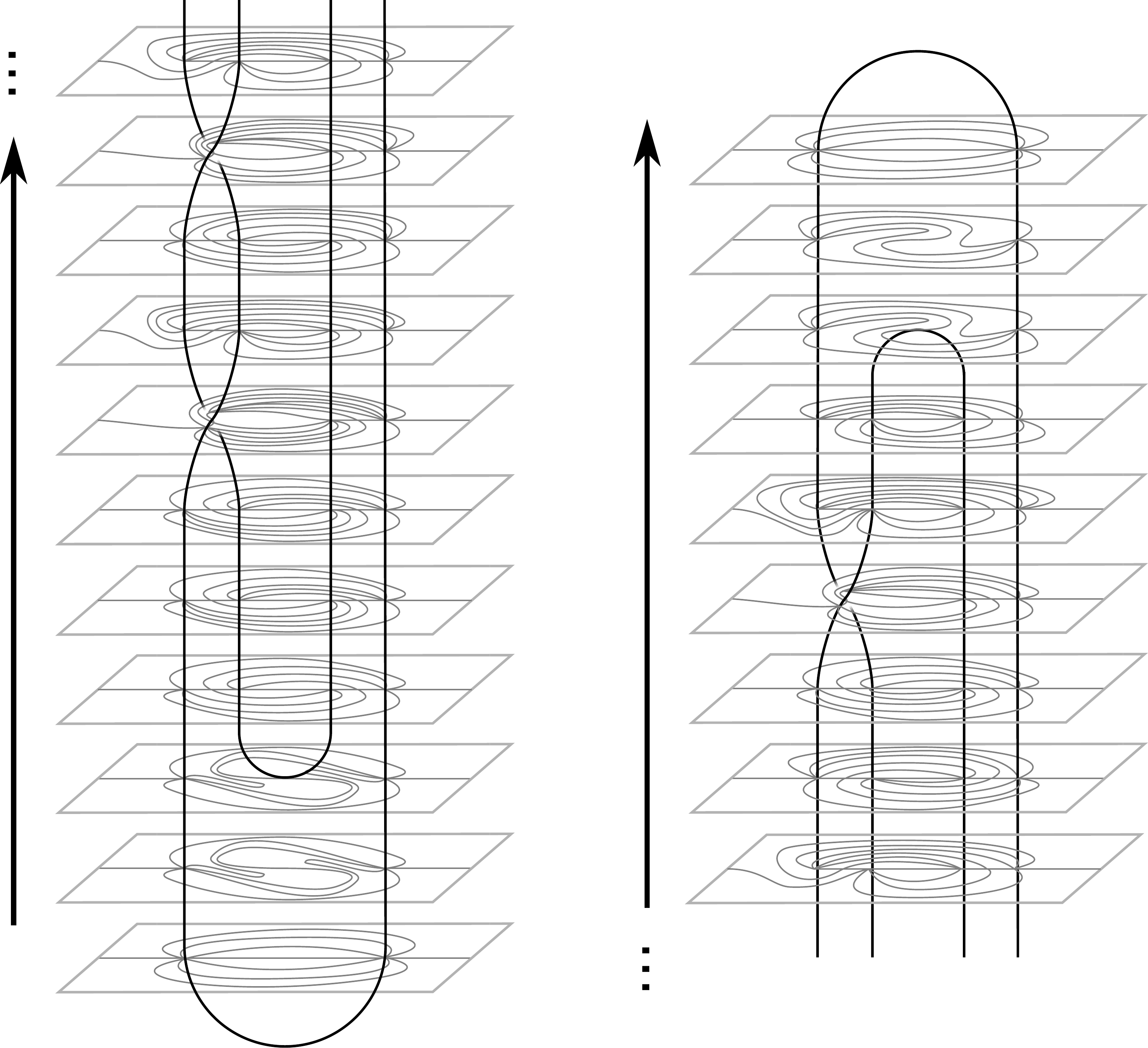}
\caption{A contour set for $f_t$ for increasing $t$, where $f_t$ is a movie of singular fibrations on the trefoil complement with total function $F$ a fibration. This movie has been constructed via the algorithm of Theorem \ref{thm:main}.}\label{fig:trefoilexample}
\end{figure}

\begin{figure}
\centering
\includegraphics[width=126mm]{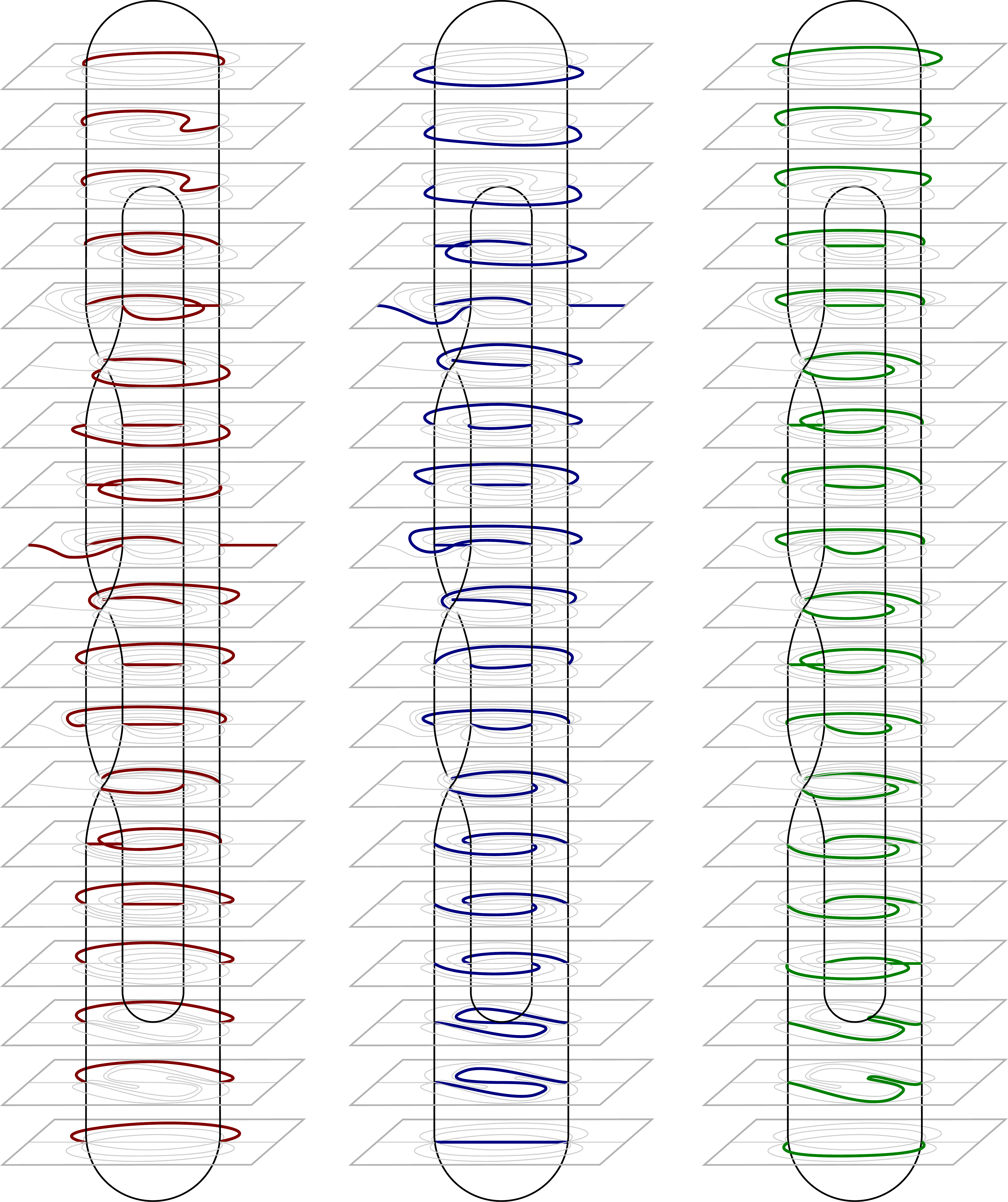}
\caption{Three fibers $F^{-1}(\theta)$ for the trefoil, constructed as in Theorem \ref{thm:main}.}\label{fig:trefoilexampleleaves}
\end{figure}

\bibliographystyle{amsplain}
\bibliography{biblio.bib}

\end{document}